\documentclass[12pt,a4paper,twoside,reqno]{amsart}

\usepackage{amssymb,amsmath,amstext,amsthm,amsfonts}
\usepackage[dvips]{graphicx}
\usepackage[ansinew]{inputenc} 
\usepackage{a4wide}
\newtheorem{maintheorem}{Theorem}

\newcommand{\cmt}{\begin{maintheorem}}
\newcommand{\fmt}{\end{maintheorem}}

\newtheorem{maincorollary}[maintheorem]{Corollary}

\newcommand{\cmc}{\begin{maincorollary}}
\newcommand{\fmc}{\end{maincorollary}}

\newtheorem{T}{Theorem}[section]
\newcommand{\cte}{\begin{T}}
\newcommand{\fte}{\end{T}}

\newtheorem{Corollary}[T]{Corollary}
\newcommand{\cco}{\begin{Corollary}}
\newcommand{\fco}{\end{Corollary}}

\newtheorem{Proposition}[T]{Proposition}
\newcommand{\cpr}{\begin{Proposition}}
\newcommand{\fpr}{\end{Proposition}}

\newtheorem{Lemma}[T]{Lemma}
\newcommand{\cle}{\begin{Lemma}}
\newcommand{\fle}{\end{Lemma}}

\theoremstyle{remark}
\newtheorem{Example}[T]{Example}
\newcommand{\cex}{\begin{Example}}
\newcommand{\fex}{\end{Example}}

\newtheorem{Remark}[T]{Remark}
\newcommand{\cre}{\begin{Remark}}
\newcommand{\fre}{\end{Remark}}

\newtheorem{Definition}[T]{Definition}
\newcommand{\cde}{\begin{Definition}}
\newcommand{\fde}{\end{Definition}}

\newcommand{\mcup}{\mbox{$\bigcup$}}

\newcommand{\ep} {\epsilon}
\renewcommand{\th} {\theta}

\newcommand{\f}{\varphi}
\newcommand{\vphi}{\varphi}

\newcommand{\w} {\omega} 
\def \RR {{\mathbb R}}

\def \ZZ {{\mathbb Z}}
\def \NN {{\mathbb N}}

\def \TT {{\mathbb T}}

\newcommand{\ra}{\rightarrow}

\newcommand{\x}{\times}

 \def \cl {{\mathcal L}}

    \def \cu {{\mathcal U}}

\newcommand{\dem}{\begin{proof}}
\newcommand{\cqd}{\end{proof}}
\newcommand{\beq}{\begin{equation}}
\newcommand{\eeq}{\end{equation}}

\newcommand{\const}{\operatorname{const}}

\newcommand{\Leb}{m}
\newcommand{\leb}{m}

\newcommand{\diam}{\operatorname{diam}}
\newcommand{\length}{\operatorname{length}}
\newcommand{\dist}{\operatorname{dist}}

\newcommand{\supp}{\operatorname{supp}}

\newcommand{\id}{\operatorname{id}}
\newcommand{\qand}{\quad\text{and}\quad}

\newcommand{\cc}{{\mathcal C}}

\newcommand{\cf}{{\mathcal F}}

\newcommand{\cp}{{\mathcal P}}

\def \cad {{\mathcal D}}
 
\title[Strong stochastic stability]{Strong stochastic stability
for non-uniformly expanding maps}

\author{Jos\'e F. Alves}
\address{CMUP\\
Rua do Campo Alegre 687, 4169-007 Porto, Portugal}
\email{jfalves@fc.up.pt}
\urladdr{http://www.fc.up.pt/cmup/jfalves}
\author{Helder Vilarinho}

\address{Universidade da Beira Interior\\
Rua Marquês d'Ávila e Bolama, 6200-001 Covilhã, Portugal}
\email{vilarinho@mat.ubi.pt}
\urladdr{http://www.mat.ubi.pt/~hsoares}
\date{\today}

\thanks{Work partially supported by FCT through CMUP. The second author was also partially supported by the FCT grant SFRH/BD/24353/2005.}



\begin{document}

\begin{abstract}
We consider random perturbations of discrete-time  dynamical systems.   We give sufficient conditions for the stochastic stability of certain classes of maps, in a strong sense.
This improves the main result in \cite{AAr03}, where it was proved the convergence of the stationary measures of the random process to the SRB
measure of the initial system in the weak$^*$ topology.
Here, under slightly weaker assumptions on the random perturbations, we obtain a stronger version of stochastic stability: convergence of the densities of the stationary measures to the density of the SRB measure of the unperturbed system in the $L^1$-norm. As an application of our results we obtain strong stochastic stability for two
classes of non-uniformly expanding maps. The first one is an open class
of local diffeomorphisms
introduced in \cite{ABV00} and the second one is the class of {\em Viana maps}.

\end{abstract}

\setcounter{tocdepth}{1}

\maketitle

\tableofcontents
\section{Introduction}

Two major goals of Dynamical Systems Theory are: to study the asymptotic
behavior
of typical orbits as time goes to infinity; and to understand how
stable that behavior is, i.e. how the behavior changes when the system is slightly
modified, or it is exposed to perturbations during time
evolution. Despite the deterministic formulation of dynamical systems, it is
easy to find examples whose evolution law is extremely simple
and whose dynamics has a high level of complexity and sensitivity to
perturbations.
This work concerns stability of systems, in a sense that we shall
precise later, in a broad class of discrete-time dynamical systems --
{\em non-uniformly expanding maps} -- when some random noise is introduced
in
the {\em deterministic} dynamics.

An well-succeeded
approach to the study of dynamical systems with complex behavior
is given by Ergodic Theory, which aims at  probabilistic description of orbits
in a measurable phase space. The existence of an invariant measure for a given
dynamics is an important fact in this context, specially if we recall Birkhoff's
ergodic
theorem, which describes time averages of observable phenomena for typical
points with respect to that measure.
However, it may happen that an invariant measure lacks of physical meaning.
\emph{Sinai-Ruelle-Bowen (SRB) measures} play a
particularly important role in this context, since they provide information
about the statistics of orbits for a \emph{large} set of initial states.
These are invariant measures which are somewhat compatible
with the reference volume measure,
when this is not preserved. For some classes of systems they can be obtained as ergodic invariant measures which are absolutely continuous with respect to the volume measure.
SRB
measures were
introduced in the 70's by  Sinai~\cite{Si72}, Ruelle~\cite{R76} and Bowen~\cite{BowR75,Bow75} for Anosov and Axiom A attractors, both in discrete and
continuous time systems. See also \cite{KS69} for uniformly expanding maps. The definition
of SRB measures has known several
formulations,  essentiality motivated by the development of the theory of
Dynamical Systems and the appearance of new examples and subjects of interest,
causing even some ambiguity on definitions in different
contemporary works. See e.g. \cite{Y02} for a compilation of related
results and historical background, and references therein.
The classes of systems studied by Sinai, Ruelle and Bowen, exhibit uniform
expansion/contraction behavior in invariant sub-bundles of the tangent bundle of
a Riemannian manifold, and statistical
properties of dynamical system with this properties were systematically
addressed
in subsequent work of many different authors. 
Systems
exhibiting expansion only in asymptotic terms have been considered in
\cite{J81},
where it was established the existence of physical measures for many quadratic
transformations of
the interval; see also \cite{CE80,BeC85,BeY92}. Related to \cite{BeC85} is the
work
\cite{BeC91} for
Hénon maps exhibiting strange attractors. Results for multidimensional
non-uniformly expanding systems appear
in \cite{V97,A00}, and motivated by these results \cite{ABV00} drawn general
conclusions for systems exhibiting non-uniformly
expanding behavior.

The introduction of random perturbations in  dynamical systems
has been addressed in
several works with slightly different means. One of the possible approaches is
 to consider
 at each
iterate a map $f_t$ close to an
original  one $f$, chosen independently
according to some probabilistic law $\theta_\ep$, where $\ep>0$ is the noise
level (for instance, in an $\epsilon$
neighborhood of the original map).
We say that $\mu_\ep$ is a {\em stationary
measure} if
\begin{equation*}
 \int\int\varphi(f_t(x))d\mu_\ep(x)d\theta_\ep(t)=\int\varphi d\mu_\ep
,
\end{equation*}
for every continuous function $\varphi:M\to \RR$.
We say that $\mu_\ep$ is a {\em physical measure} if for a set with positive
Lebesgue measure of initial states $x\in M$ we have
\begin{equation*}
\lim_{n\to+\infty}\frac1n\sum_{j=0}^{n-1}
\varphi((f_{t_{j-1}}\circ\dots\circ f_{t_0})(x))=\int\varphi
d\mu_\ep ,\end{equation*} for every continuous $\varphi:M\to\RR$ and almost
all sequence $(t_0,t_1,\ldots)$ with respect to the product measure
$\theta_\ep^\NN$.
Physical measures for random perturbations
play an equivalent role to that of SRB measures in the deterministic context.
In
order to distinguish them in the deterministic and random
perturbation contexts, we shall refer to
{physical measures} only in the random perturbation setting and to SRB measures
in the deterministic setting.

Stochastic stability is a rather vague
notion, depending on the nature of the systems under consideration, but it tries to reflect that the introduction of small random noise
affects just slightly the statistical description of the dynamical system.  We call a
system   {\em stochastically
stable}
if the stationary physical measures converge in the weak$^*$ topology to some
SRB measure, as $\ep$ goes to zero, and  {\em strongly
stochastically
stable} if the
convergence is with respect to the densities (if they exist)  in
the $L^1$-norm.
We can also formulate random perturbations and stochastic stability in terms of Markov chains.   We refer to \cite{Ki86,Ki88} for a background and treatment of the
topic. For stochastic stability results see  \cite{Y86,BaY93,BaKS96,Ba97} for uniformly hyperbolic systems,
\cite{KatKi86,BeY92,BaV96,Me00} for  non-uniformly expanding interval
maps, \cite{BeV06} for Hénon-like maps and \cite{AAV07} for partially hyperbolic attractors. For related topics see  e.g.
\cite{CoY05} for an analysis of SRB
measures as zero-noise limits, and \cite{Ar00} for an important contribution to
the
stochastic part of a conjecture by
Palis~\cite{Pa00}.

Stochastic
stability was established in \cite{AAr03} for a general class of  multidimensional non-uniformly
expanding maps  in
the weak sense. The main goal of the present work is to improve that
result in \cite{AAr03} to
strong stochastic stability, and this actually happens to hold in a more general framework.
In particular, no nondegeneracy conditions  as in \cite[Section~3]{AAr03} are imposed.
Our main result is stated in Theorem~\ref{thA} and is formulated in a way that
enables us to use the result in several situations and examples,
and can be a useful tool in the analysis of stochastic properties of dynamical
systems with non-uniform expanding behavior.

\subsubsection*{Overview} This work is organized in the following way. In the remaining of this Introduction we present formally the main
definitions and results on the strong stochastic stability for
non-uniformly expanding maps,  allowing the
presence of  critical set.  Sections~\ref{se.measures} and \ref{se.stability}
are devoted to prove Theorem~\ref{thA}.
In Section~\ref{se.measures} we follow initially
some ideas from \cite{BaBeM02}  on a random version of Young towers to
construct an absolutely continuous stationary probability measure and  prove that
this stationary measure is ergodic
and therefore unique. This approach is based on Theorem~\ref{random Markov structure} where we obtain random induced schemes for the stochastic perturbations under consideration. This theorem is a stochastic version of the main result in \cite{A04}.
The proof of Theorem~\ref{random Markov structure} is left to Section~\ref{se.monitoring} and it extends ideas from \cite{ABV00,ALP05} on deterministic
  non-uniformly expanding maps to the present situation. It also uses previous material from  \cite{AAr03} which, on  its own, extends results from \cite{ABV00} to
the random
situation. In Section~\ref{se.stability} we prove the strong stochastic stability, inspired in the approach  of
\cite{AV02}, where strong statistical stability is achieved.
In Section~\ref{se.examples}
we present applications of our main result to two classes of examples that fit our assumptions and for which we
obtain the strong stochastic stability. The first example is an open
class of local diffeomorphisms
introduced in \cite{ABV00}, and the second one {\em Viana maps}, an open class of maps with critical sets introduced in \cite{V97}.
This improves the weaker form of stochastic stability  proved in \cite{AAr03} for both examples.

\subsection{Non-uniformly expanding maps}\label{ch.NUE}
 Let $M$ be a compact boundaryless manifold endowed with a normalized
volume measure $m$ that we call Lebesgue measure. Let $f\colon M\to M$ be a $C^2$ local diffeomorphism in the whole
manifold  except, possibly, in a set $\cc\subset M$ of critical/singular points. This
set $\cc$ may be taken as a set of points where the derivative
of $f$ is not an isomorphism or simply does not exist.

 \cde\label{d.nd} We say that a critical/singular set $ \cc $ is
{\em non-degenerate} if it has zero Lebesgue measure
and the following conditions hold:
\begin{enumerate}
  \item There are constants $B>1$ and $\beta>0$ such that for every $x\in
 M\setminus\cc$
\begin{enumerate}
 \item[(c$_1$)]
\quad $\displaystyle{\frac{1}{B}\dist(x,\cc)^{\beta}\leq \frac
{\|Df(x)v\|}{\|v\|}\leq B\dist(x,\cc)^{-\beta}}$ for all $v\in T_x
M$.
\end{enumerate}
  \item For every
$x,y\in M\setminus \cc$ with $\dist(x,y)<\dist(x,\cc)/2$ we have
\begin{enumerate}
\item[(c$_2$)] \quad $\displaystyle{\left|\log\|Df(x)^{-1}\|-
\log\|Df(y)^{-1}\|\:\right|\leq
\frac{B}{\dist(x,\cc)^{\beta}}\dist(x,y)}$;
 \item[(c$_3$)]
\quad $\displaystyle{\left|\log|\det Df(x)|- \log|\det
Df(y)|\:\right|\leq \frac{B}{\dist(x,\cc)^{\beta}}\dist(x,y)}$.
 \end{enumerate}
\end{enumerate}

 \fde
The first condition
 says that
 $f$  {\em behaves like a power of the distance}
 to $ \cc $ and the last two conditions say that the functions $  \log|\det Df| $ and $ \log \|Df^{-1}\|
$ are \emph{locally Lipschitz} in $ M \setminus \cc$,
with the Lipschitz constant depending on the distance to $\cc$.
 Given $\delta>0$ and $x\in M\setminus\cc$ we define the
{\em $\delta$-truncated distance\/} from $x$ to $\cc$ as
\begin{equation*}
\label{e.truncate} \dist_\delta(x,\cc)=\left\{
\begin{array}{ll} 1, & \text{if } \dist(x,\cc)\ge\delta \\
                  \dist(x,\cc), & \text{otherwise}.
\end{array}\right.
\end{equation*}
\
\cde\label{def.NUE} Let $f\colon M\to M$ be a $C^2$ local
diffeomorphism outside a non-degenerate critical set $\cc$. We say
that $f$ is {\em non-uniformly expanding on a
set}\index{non-uniformly expanding!on a set} $H\subset M$ if the
following conditions hold:
\begin{enumerate}
\item there is ${a_0}>0$ such that for each $x\in H$
 \begin{equation}\label{NUE}
    \limsup_{n\to{+\infty}}\frac{1}{n}\sum_{j=0}^{n-1}
    \log \|Df({f^{j}(x))}^{-1}\|<-{a_0};
\end{equation}
\item  for every ${b_0}>0$ there exists $\delta>0$ such that
for each $x\in H$
\begin{equation} \label{e.faraway1}
    \limsup_{n\to+\infty}
\frac{1}{n} \sum_{j=0}^{n-1}-\log \dist_\delta(f^j(x),\cc)
<{b_0}.
\end{equation}
\end{enumerate}
We will refer to the second condition above by saying that the
orbits of points in  $H$ have {\em slow recurrence}\index{slow
recurrence} to  $\cc$.
 The case
 $\cc=\emptyset$ may  also be considered, and in
 such case the definition reduces to the
 first condition. A map is said to be
{\em non-uniformly expanding}\index{expanding!non-uniformly}  if it
is non-uniformly expanding on a set of full Lebesgue measure.
 \fde

\subsection{Random perturbations}\label{ch.random}
 The idea of random perturbations
is to
replace the original deterministic obits
by {\em random orbits} generated by an independent and
identically distributed random choice of
map at each iteration. To be more precise, given a dynamical system $f:M\to M$, consider a  family $\mathcal F$ of maps from $M$ to $M$ endowed with some metric, a metric space $T$ and a continuous map
$$
 \begin{array}{rccl}
 \Phi
:& T &\longrightarrow&  \cf\\
 & t &\longmapsto & \Phi(t)=f_t
 \end{array}
 $$
such that $f=f_{t^*}$ for some  $t^*\in T$. Moreover, let $(\theta_\ep)_{\ep>0}$ to be a family of Borel probability measures in $T$. We consider the product space $T^\NN$ and product probability measure $\theta_\ep^\NN$   on $T^\NN$. We will refer to such a pair $\{\Phi,(\th_\ep)_{\ep>0}\}$ as a {\em random
perturbation} of $f$.

 For a {\em realization}
$\omega=(\omega_0,\omega_1,\ldots)\in T^{\NN}$ and $n\ge0$ we define
\begin{equation*}
f_{\omega}^n(x)=\left\{\begin{array}{ll}
x& \text{ if }\quad  n=0,\\
 (f_{\omega_{n-1}}\circ\dots\circ f_{\omega_1}\circ f_{\omega_0})(x)& \text{ if
}\quad  n>0.\\
\end{array}\right.
\end{equation*}
Given $x\in M$ and  $\w\in T^\NN$ we call the sequence
$\big(f_\omega^n(x)\big)_{n\in\NN}$ a {\em random
orbit}\index{random orbit} of $x$. 

\cde
A measure $\mu_\ep$ on the Borel sets of $M$ is called a
{\em stationary measure} for $\{\Phi,(\th_\ep)_{\ep>0}\}$ if
$$\iint(\varphi\circ f_t)(x)
\,d\mu_\ep(x)d\theta_\ep(t)=\int\varphi\,d\mu_\ep,$$
for all $\varphi:M\to\RR$ continuous.\fde

If there is no confusion we will refer such a measure $\mu_\ep$ as a
stationary measure for
$f$.

\subsection{Non-uniform expansion on random orbits}
Consider a random perturbation $\{\Phi,(\theta_\ep)_{\ep>0}\}$ of a non-uniformly expanding map $f$ such that, with respect the metric on $\cf$,
$$\supp(\theta_\ep)\rightarrow \{t^*\},\quad\text{as}\quad
\ep\to 0.$$ Due to the presence of the critical set, we will restrict the class of
perturbations we are going to
consider for maps with critical sets: we take all the maps $f_t$
with the same critical set $\cc$ by imposing that
 \begin{equation}\label{e.perturbation}
 Df_{t}(x)=Df(x), \quad\mbox{for every $x\in M\setminus\cc$ and $t\in T$}.
 \end{equation}
 This may be implemented, for instance, in parallelizable manifolds (with an
additive
group structure, e.g. tori $\mathbb T^d$ (or cylinders $\TT^{d-k}\x
\RR^k$), by considering     $T=\{t\in\RR^d: \|t\|\leq\ep_0\}$
for some $\ep_0>0$, and
taking $f_t=f+t$, that is, adding at each step a random noise to the
unperturbed
dynamics.

\cde\label{defNUEOA}  We say that $f$ is
\emph{non-uniformly expanding on random orbits}  if the following conditions hold, at
least for small $\ep>0$:
\begin{enumerate}
  \item  there is ${a_0}>0$ such that for
$\th_\ep^\NN\times m$ almost every $(\omega,x)\in T^\NN\times M$
 \begin{equation} \label{NUEOA}
\limsup_{n\ra +\infty}\frac{1}{n}
\sum_{j=0}^{n-1}\log\|Df(f_{\omega}^j(x))^{-1}\|< -{a_0};
 \end{equation}
 \item  given any small ${b_0} >0$ there is $\delta
>0$ such that for $\th_\ep^\NN\times m$ almost every $(\omega,x)\in T^\NN\times
M$
\begin{equation}
 \label{RecLentOA}
\limsup_{n\ra +\infty}\frac{1}{n}
\sum_{j=0}^{n-1}-\log\dist_\delta(f^j_{\omega}(x), \cc)< {b_0}.
 \end{equation}
\end{enumerate}
 \fde
 When $\cc=\emptyset$
we naturally disregard the second condition in the definition
above. In this case
we can remove assumption
\eqref{e.perturbation} and replace \eqref{NUEOA} by the following condition:
there is
$a_0>0$ such that for $\theta_\ep^\NN\x\leb$ almost every $(\w,x)\in T^\NN\x M$
 \begin{equation}\label{NUEOA2}
    \limsup_{n\to{+\infty}}\frac{1}{n}\sum_{j=0}^{n-1}
    \log \|Df_{\sigma^j(\w)}({f_\w^{j}(x))}^{-1}\|<-{a_0}.
\end{equation}

Condition \eqref{NUEOA} implies
that for $\theta_\ep^\NN$ almost every $\omega\in T^\NN$, the
\emph{expansion time} function
\begin{equation*}\label{de.exptime}
\mathcal E_\omega(x) = \min\left\{N\ge1\colon
\frac1n\sum_{j=0}^{n-1} \log \|Df({f_\omega^{j}(x))^{-1}}\| \leq
-{a_0}, \text{ for all $n\geq N$}\right\}
\end{equation*}
is defined and finite Lebesgue almost everywhere in $ M $. 

According to Remark~\ref{r.strongran2}, condition \eqref{RecLentOA} is not
needed in all its strength. Actually, it is enough that it holds for suitable ${b_0}>0$, and
$\delta>0$ chosen in such a way that the proof of
Proposition~\ref{pr.hyperbolic1} works. In view of this,
for $\theta_\ep^\NN$ almost every $\omega\in T^\NN$ we can  define
the \emph{recurrence time} function Lebesgue almost everywhere in
$ M $,
\begin{equation*}\label{de.slowrec}
\mathcal R_\omega(x) = \min\left\{N\ge 1: \frac1n\sum_{j=0}^{n-1}
-\log \dist_\delta(f_\omega^j(x),\cc) \leq {b_0} , \quad
\text{for all }n\geq N\right\}.
\end{equation*}
We
introduce the {\em tail set (at time
$n$)}\index{tail set}
\begin{equation}\label{tailset}
    \Gamma_\omega^n=\big\{x: \mathcal E_\omega(x) > n \ \text{ or } \ \mathcal R_\omega(x) > n \big\}.
\end{equation}
This is the set of points in $ M $ whose random orbit at time $ n
$ has not yet achieved either the uniform exponential growth of
derivative or the slow recurrence given by
conditions~\eqref{NUEOA} and~\eqref{RecLentOA}. If the critical
set is empty, we simply ignore the recurrence time function and
consider only the expansion time function in the definition of
$\Gamma_\omega^n$.

\subsection{Strong stochastic stability}

It is known
 that a non-uniformly
expanding map $f$ admits a finite number of absolutely continuous ergodic
invariant probability measures  (\emph{SRB}
measures); see \cite{ABV00}. Moreover, if $f$ is also topologically transitive, then it has a unique SRB probability measure $\mu_f$; see \cite{A03}.
We state now our main result which asserts the existence of a unique absolutely continuous ergodic
stationary probability measure and the strong stochastic stability for
non-uniformly expanding maps, meaning convergence in the $L^1$-norm of the
density of the stationary measure to the density the unique SRB
probability measure.

\cmt\label{thA} Let $f$ be a transitive non-uniformly expanding
map and non-uniformly expanding on random orbits, for which exist
$p>1$ and $C>0$ such that $\Leb(\Gamma_\omega^n)<Cn^{-p}$
for $\theta_\ep^\NN$ almost every $\omega\in T^\NN$. Then
\begin{enumerate}\item[1.] if $\ep>0$ is small enough, then $f$
admits a unique
absolutely continuous ergodic stationary probability
measure
$\mu_{\ep}$; \item[2.]  $f$ is
strongly
stochastically stable:
$$\lim_{\ep\to0}\left\|\frac{d\mu_\ep}{d\leb}-\frac{d\mu_f}{d\leb}\right\|_1=0.$$
\end{enumerate}\fmt

This theorem improves the main result in \cite{AAr03}, where stochastic stability was
established the  in the weak sense (convergence of $\mu_\ep$ to $\mu_f$ in the weak$^*$ topology). Furthermore,  our arguments for the strong stochastic stability can be carried out with no
extra assumptions on the probabilities $\theta_\ep$   as in \cite[Section~3]{AAr03}. 
%

\section{Measures on random perturbations}\label{se.measures}

Throughout this section we prove the first item of Theorem \ref{thA}. Our strategy includes an extrapolation to a two-sided random perturbation setting that we introduce in Section~\ref{generalities} as well as known results relating both one-sided and two-sided random perturbations. In Section \ref{ris} we use this two-sided setting to construct random induced Gibbs-Markov strucures. In Section \ref{ssm} we construct suitable induced measures and use them to obtain an absolutely continuous stationary probability measure that we prove to be ergodic and unique in Section \ref{eau}.

\subsection{Generalities on stationary measures}\label{generalities}
We introduce the two-sided random perturbations, considering also the past of the realizations.
Similarly to the one-sided case, we consider the product space $T^\ZZ$ and the probability product measures $\theta_\ep^\ZZ$.
We define the {\em two-sided skew-product} map as
 \begin{equation*}
\begin{array}{rccc} S: & T^\ZZ\times M &\longrightarrow &
T^\ZZ\times M\\
 & (\omega, z) &\longmapsto & \big(\sigma(\omega),f_{\omega_0}(z)\big),
\end{array}
 \end{equation*}
 where $\sigma\colon T^\ZZ \to T^\ZZ$ is the left shift map.
  It is well known that a Borel probability measure $\mu^*$ in $T^\ZZ\times M$ invariant by  $S$ (in the usual deterministic sense) is characterized by an essentiality unique
disintegration $d\mu^*(\omega,x)=d\mu_{\omega}(x)
d\theta_\ep^{\ZZ}(\omega)$ given by a family $\{\mu_\w\}_{\w}$ of
{\it sample measures} on $M$ with the following properties:
\begin{enumerate}
  \item  $\omega\mapsto\mu_\omega(B)$ is $\theta_\ep^\ZZ$-measurable, for each Borel set $B\subset M$;
  \item  $B\mapsto\mu_\w(B)$ is a Borel probability measure in $M$, for each $\theta_\ep^\ZZ$ almost every $\w$;
  \item ${f_{\omega}}_*\mu_\omega=\mu_{\sigma(\omega)}$, for
$\theta_\ep^{\ZZ}$ almost every $\w$.
\end{enumerate}
 The relation between $\mu^*$ and the family
of
sample measures can be expressed as
$$\mu^*(A)=\int\mu_\w(A_\w)\,d\theta_\ep^\ZZ(\w),$$
where $A$ is a Borel subset of $T^\ZZ\times M$ and $A_\w=\{x\in M: (\w,x)\in
A\}$.

Given $\w=(\ldots,\w_{-1},\w_0,\w_1,\ldots)\in T^\ZZ$ we define the {\em  future} of $\w$ as
$\w^+=(\w_0,\w_1,\ldots)$ and the {\em  past} of $\w$ as
$\w^-=(\ldots,\w_{-2},\w_{-1})$. We consider the projection map \begin{equation*}\begin{array}{rccc}
\pi:&T^\ZZ\times M&\to& T^\NN\times M\\ &\pi(\w,x)&\mapsto&(\w^+,x).\end{array}
 \end{equation*}
We say that a Borel measure $\mu^*$ on $T^\ZZ\times M$ is a
{\it Markov measure} if for
$\theta_\ep^\ZZ$ almost every $\w\in T^\NN$ the corresponding sample measure $\mu_w$ depends only on the past $\w^-$ of $\w$.

\begin{Proposition}\label{relationmeasures}
The stationary probabilities $\mu_\ep$ for
$\{\Phi,(\theta_\ep)_{\ep>0}\}$ are
in a one-to-one correspondence with the $S$-invariant Markov probabilities
$\mu^*$, with that corres\-pon\-den\-ce being given by
$$\mu^*\mapsto \mu_\ep:=\int\mu_\w \,d\theta_{\ep}^{\ZZ}(\w)\quad\text{ and
}\quad \mu_\ep\mapsto\mu^*:=\lim_{n\to+\infty}
S^n_*(\theta_\ep^\ZZ\x\mu_\ep).$$
Moreover, for given stationary probability measure $\mu_\ep$, the corresponding
$\mu^*$
can
be recognized as the unique $S$-invariant probability measure such that
$\pi_*\mu^*=\theta_{\ep}^\NN\times\mu_\ep$.
 \end{Proposition}
 \begin{proof} See \cite{Arn98}.
 \end{proof}
 From now on, we refer for $\mu_\ep$ and $\mu^*$ to be the corresponding stationary and
Markov probability measures, respectively.
We define the {\em one-sided skew-product} map by
\begin{equation*}
\begin{array}{rccc} S^+: & T^\NN\times X &\longrightarrow &
T^\NN\times X\\
 & (\omega, z) &\longmapsto & \big(\sigma^+(\omega^+),f_{\omega_0}(z)\big),
\end{array}
 \end{equation*}
 where $\sigma^+\colon T^\NN \to T^\NN$ is the one-sided left shift map. It is easy to see that $S^+\circ\pi=\pi\circ S$.

\begin{Proposition}
The following conditions are equivalent:
 \begin{itemize}
  \item[i)] $\mu_\ep$ is a stationary probability measure.
  \item[ii)] $\mu^*$ is $S$-invariant.
  \item[iii)] $\theta_{\ep}^\NN\times\mu_\ep$ is $S^+$-invariant.
 \end{itemize}
 \end{Proposition}
 \begin{proof}
  See \cite{O83} for the equivalence between {\em i)} and {\em iii)}. \end{proof}

\begin{Definition}\label{random inv}
 A set $A\subset M$ is {\em random invariant} if for
$\mu_\ep$ almost every $x\in M$ we have

 \begin{tabular}{rcl}
 $x\in A$ &$\implies$& $f_t(x)\in A$, for  $\theta_\ep$  almost every $t$;\\
  $x\in M\setminus A$&$ \implies$&$ f_t(x)\in M\setminus A$, for  $\theta_\ep$ almost every $t$.
 \end{tabular}
\end{Definition}

\begin{Definition} A stationary measure $\mu_\ep$ is {\em ergodic} if for
every random invariant set $A$ we have
 $\mu_\ep(A)=0$ or $\mu_\ep(A)=\mu_\ep(M)$.
\end{Definition}

\begin{Proposition}\label{equiv.ergodic}
 The following conditions are equivalent:
 \begin{itemize}
  \item[i)] $\mu_\ep$ is ergodic.
  \item[ii)] $\theta_{\ep}^\NN\times\mu_\ep$ is $S^+$-ergodic.
  \item[iii)] $\mu^*$ is
$S$-ergodic.
 \end{itemize}
 \end{Proposition}
 \begin{proof} See \cite{Ki86} for the equivalence between {\em i)} and {\em ii)} and \cite{LQ95} for the equivalence between {\em ii)} and {\em iii)}
 \end{proof}

\subsection{Random inducing schemes} \label{ris}

From now on we will consider the two-sided random perturbations scheme. We define non-uniformly expansion in random orbits similarly to the previous one-sided definition, just considering two-sided realizations in conditions \eqref{NUEOA} (or \eqref{NUEOA2} if $\cc=\emptyset$) and \eqref{RecLentOA}. Analogously, we define the functions $\mathcal{E}_\w, \mathcal{R}_\w$ and the tail set $\Gamma_\w^n$ for $\w\in T^\ZZ$ and $n\ge 0$. It is easy to see that if we assume the hypothesis of the main theorem with respect to the one-sided random perturbations they still hold in the two-sided environment.

We set $\Omega_\ep$ as the $\theta_\ep^\ZZ$ full
measure subset of realizations $\w\in T^\ZZ$ for which conditions
\eqref{NUEOA} and \eqref{RecLentOA} are satisfied for all
$\sigma^k(\w), k\in\ZZ$, and Lebesgue almost
every $x\in M$. Note that if $f$ is itself a non-uniformly expanding map then $\w^*=(\ldots,t^*,t^*,t^*,\ldots)$
belongs to~$\Omega_\ep$.

\cde\label{def random markov
structure}
We say that $\w\in T^\ZZ$ induces a piecewise expanding
\emph{Gibbs-Markov}s map $F_{\omega}$ in a ball
$\Delta\subset M$ if there is a countable
partition $\cp_{\omega}$ of a full
Lebesgue measure subset $\mathcal{D}$ of
$\Delta$ and a \emph{return time function}
$\displaystyle R_\omega:\mathcal{D}\to\NN$, constant in each
$U_{\omega}\in\cp_{\omega}$, such that the map
$F_{\omega}(x)=f_\w^{R_\w(x)}(x):\Delta\to\Delta$ verifies:
\begin{enumerate}
\item{\em Markov:} ${F}_{\omega}$ is a $C^2$ diffeomorphism
        from each $U_{\omega}\in\cp_{\omega}$ onto $\Delta$.
\item {\em Expansion:} \label{random induced
expansion}
    there is $0<\kappa_\w<1$ such that for $x$
    in the interior of $U_{\omega}\in\cp_{\omega}$ $$\|D
{F}_{\omega}(x)^{-1}\| <\kappa_w.$$
\item {\em Bounded distortion:\/}\label{random bounded distortion}
    there is some constant $K_\w>0$ such that for every
    $U_{\omega}\in\cp_{\omega}$ and $x,y\in U_{\omega}$
    \[
    \log\left|\frac{\det D{F}_{\omega}(x)}{\det D{F}_{\omega}(y)}\right| \leq
K_\w
    \dist({F}_{\omega}(x), {F}_{\omega}(y)).
    \]
 \end{enumerate}
\fde
\noindent For simplicity of notation we shall write  $\{R_\omega>n\}$
for the set
$\{x\in\Delta:R_\omega(x)>n\}$.

\cte\label{deterministic GB}
Let $f:M\to M$ be a transitive non-uniformly expanding map. The realization $\w^*$, associated to the deterministic dynamics $f$, induces a piecewise expanding
Gibbs-Markov map
${F}:\Delta\to\Delta$, for some ball
$\Delta\subset M$.
\fte
\dem
See \cite{ALP05}.
\cqd

\cre\label{remark muf}
It is well known that a Gibbs-Markov map $F$ admits a unique absolutely continuous ergodic invariant  probability measure $\mu_F$; see e.g. \cite{Y99}. From this fact one easily deduces that the measure \begin{equation*}
\tilde\mu_f =
\displaystyle\sum_{j=0}^{{+\infty}}{f^j}_{\ast}\left(\mu_F\vert
\{R>j\}\right),
\end{equation*}
is absolutely continuous ergodic and invariant by the map $f$. The integrability of  the return time function $R$ with respect to $\leb$
implies that the measure $\tilde\mu_f $ is finite.  In such case we denote by  $\mu_f$ the normalization of $\tilde\mu_f$.
\fre

In what follows, $\Delta$ is the ball given by Theorem~\ref{deterministic GB}. The next theorem ensures that almost all realizations
induce
 piecewise expanding
Gibbs-Markov maps with some uniformity on the constants. Most of the
auxiliary results we use to prove this theorem can be
obtained by mimicking the deterministic ones in \cite{ALP05}, being that some of them
have already been extended to random perturbations in \cite{AAr03}.
Nevertheless, we describe in
detail their proofs in Section~\ref{se.monitoring}, in order to easily track the
extension to random
perturbations and
monitor a certain uniformity on random orbits, which is essential for our
purposes.

\cte\label{random Markov structure}
Let $f:M\to M$ be a transitive non-uniformly expanding map and non-uniformly expanding on random orbits. If $\ep>0$ is small enough
then
 \begin{enumerate}
  \item every
$\omega\in \Omega_\ep$ induces a piecewise expanding
Gibbs-Markov map $F_{\omega}$ in $\Delta\subset M$;
\item if there exist $p>0$, $C>0$ such that
$\Leb(\Gamma_\omega^n)<Cn^{-p}$ for every $\omega\in\Omega_{\epsilon}$,
  then there exists $C'>0$ such that for every
$\w\in\Omega_\ep$
 the return time function
satisfies
\begin{equation}\label{uniform tail decay}
\leb(\{R_\w>n\})\leq C'n^{-p}.
\end{equation}

\end{enumerate}
\fte

As we shall see latter, the proof of this theorem also gives that the
following uniformity conditions hold:
\begin{itemize}
\item[(U1)] Given integer $N>1$ and ${\gamma}>0$, then for $\ep>0$ is sufficiently small and $j=1,2,\ldots,N$
$$
m\left(\{R_{\sigma^{-j}(\omega)}=j\}\triangle\{R_{\sigma^{-j}(\tau)}
=j\}\right)\leq
{\gamma},\quad \forall \omega,\tau\,\in\Omega_\ep,$$
where $\triangle$ stands for the symmetric difference of two sets.

\item[(U2)] Given $\ep>0$ sufficiently small, then for every
$\w\in\Omega_\ep$, the constants $K_\w$ and
$\kappa_\w$
in the definition of induced piecewise expanding Gibbs-Markov map can be chosen uniformly. We will
refer to them as $K>0$ and $\kappa>0$, respectively.
 \end{itemize}

 From now one we assume the hypothesis of Theorem \ref{thA} and we consider $\ep>0$ sufficiently small so that  Theorem~\ref{random Markov structure} and conditions (U1) and (U2)
hold.

\subsection{Sample and stationary measures}\label{ssm}

We start defining a random induced dynamical system. This is the main motivation for the introduction of the two-sided random perturbations. Let us consider disjoint
copies $\Delta_\omega$ of $\Delta$,
associated to an $\omega\in\Omega_\ep$, and their partitions $\cp_\w$.  For $x\in\Delta_\w$ we
define
$F_\omega(x)=f_\omega^{R_\omega(x)}(x)$ and the dynamics consists in hopping
from $x\in\Delta_\omega$ to $F_\w(x)\in\Delta_{\sigma^{R_\omega(x)}(\omega)}$.
However, we also can keep regarding
this as a dynamical system in
$\Delta$. We refine recursively $\mathcal P_\w$ on $\Delta_\w$ with the partitions associated to
the images of each element of $\mathcal P_\w$: $$\mathcal
P_w^{(n)}=\bigvee_{j=0^{\phantom{j}}}^{n}\bigvee_{k\in
L_{\w,j}}(F_\w^j)^{-1}{\mathcal P}_{\sigma^k(\w)}$$
where $L_{\w,j}=\{k\in\NN_0:
F_\w^j(\Delta_\w)\cap\Delta_{\sigma^k(\w)}\neq\emptyset\}$.

  Our aim now is to prove that for each $\omega\in\Omega_\ep$ there is an
absolutely
continuous measure $\nu_{\omega}$ defined on $\Delta$ with some invariance property. Moreover, the density of
$\nu_{\omega}$ with respect to the Lebesgue measure will belong to
a {Lipschitz}-{type} space:
\begin{equation*}
\mathcal{H} = \left\{\varphi:\Delta\to\mathbb{R}\vert ~ \exists
K_\varphi>0,\, \left|{\varphi(x)}-{\varphi(y)}\right|\leq K_\varphi d(x,y)\,\,
~~\forall
x,y \in \Delta \right\}.
\end{equation*}
Given a measurable set
$A\subset\Delta_{\omega}$ we define
$$(F^{-1})_{\omega}(A)=\bigsqcup_{n\in\NN}\left\{x\in\Delta_{\sigma^{-n}
(\omega) } : R_{\sigma^{-n}(\omega)}(x)=n \qand
F_{\sigma^{-n}(\omega)}(x)\in A\right\}$$
and define $[(F^{j})^{-1}]_{\omega}(A)$ by induction.
Given a family
$\{\nu_{\sigma^{-n}(\omega)}\}_{n\in\NN}$ of measures on
$\displaystyle\bigsqcup_{n\in\NN}\Delta_{\sigma^{-n}(\omega)}$ we set
\begin{equation*}\label{pseudo
pf}{(F^{j})_{\omega}}^{*}\{\nu_{\sigma^{-n}(\omega)}\}_{n\in\NN}(A)=\sum_{n\in\NN}
\nu_{\sigma^{-n}(\omega)}([(F^{j})^{-1}]_{\omega}(A)\cap\Delta_{\sigma^{-n}
(\omega) } )
\end{equation*}
Here we use $*$ superscript  to distinguish this push-forward from the one for deterministic systems, whose notation is usually $*$ subscript.

\cte\label{exist.mu.induced} For every $\omega\in
\Omega_\epsilon $ there is an absolutely continuous
 finite measure $\nu_\w$ on $\Delta$ such that
${(F)_\w}^{*}\{\nu_{\sigma^{-n}(\omega)}\}_{n\in\NN}=\nu_\w$ and $\rho_{\omega}=d\nu_{\omega}/d\leb\in\mathcal{H}$.
Moreover, there is a constant $K_1>0$ such that
${K_1}^{-1}\leq
\rho_{\omega}\leq K_1$ for all $\w\in\Omega_\ep$.\fte

\begin{proof}  
Let $m_0$ be the probability measure $(m\vert{\Delta})/m(\Delta)$ on
$\Delta$ and
set $\{m_0\}_{n\in\NN}$ as the family of measures on
$\bigsqcup_{k\in\NN}\Delta_{\sigma^{-k}(\omega)}$ so that $m_0$ is the
measure on each $\Delta_{\sigma^{-k}(\w)}$.
For every
$A\subset\bigsqcup_{k\in\NN}\Delta_{\sigma^{-k}(\omega)}$,
with $A\subset [(F^j)_\w]^{-1}(\Delta_\w)$ and
$A\in{\mathcal P}_{\sigma^{-n}(\w)}^{(j)}$, for some $n\in\NN$, we define on
$\Delta_\w$ the function
$$
\rho_\w^{j,A}=\frac{d}{d\leb_0}(F_{\sigma^{-n}(\w)}^j)_{*}({m}_0\vert A).
$$
Let $x,y\in\Delta_{\omega}$ be arbitrary points, and let $x',y'\in
A$ be such that $x'\in [(F^{j})^{-1}]_{\omega}(x)$ and $y'\in
[(F^{j})^{-1}]_\w(y)$,  so that
$x',y'\in\Delta_{\sigma^{-n}(\omega)}$. For the
decreasing sequence $n=n_0>\ldots>n_j=0$ given by
$$ n_l=n_{l-1} -
R_{\sigma^{n_{l-1}}(\omega)}(F_{\sigma^{-n}(\omega)}^{l-1}(x')),\quad\text{ for
}1\leq l\leq j,$$
 we find that
$F_{\sigma^{-n}(\omega)}^{l}(x'), F_{\sigma^{-n}(\omega)}^l(y')$
lies in the same element $U_{\sigma^{-n_{l}}(\omega)}$ of
$\mathcal{P}_{\sigma^{-n_{l}}(\omega)}$, for $0\le l<j$.
By Theorem \ref{random Markov structure} (recall items \ref{random
induced expansion}.  and \ref{random bounded distortion}. in Definition \ref{def
random markov
structure} and ($U2$))
\begin{equation*}
\log\frac{\rho_\w^{j,A}(y)}{\rho_\w^{j,A}(x)}=\log\frac{|\det D
F_{\sigma^{-n}(\omega)}^j(x')|}{|\det D
F_{\sigma^{-n}(\omega)}^j(y')|}=\sum_{l=0}^{j-1}\log\left|\frac{\det D
F_{\sigma^{-n_l}(\omega)}({ F_{\sigma^{-n}(\omega)}^l}(x'))}{\det D
F_{\sigma^{-n_l}(\omega)}({ F_{\sigma^{-n}(\omega)}^l}(y'))}\right|\leq
K_1'\dist(x,y),
\end{equation*}
with $K_1'=K\frac1{1-\kappa}$, which is uniform in $\omega$, $j$ and $A$.
The
sequence
 $$\rho_{\omega,n}=\frac{d}{d\leb_0}\left(\frac1n \sum_{j=0}^{n -1}
{(F^j)_{\omega}}^{*}\{m_0\}_{n\in\NN} \right)$$
is a linear combination of terms as $\rho_\w^{j,A}$ so that one has
 $ \rho_{\omega,n}(x)\leq \exp(K_1'2\delta_0)\rho_{\omega,n}(y)$ for all $x,y$ in
$\Delta$, where
$\delta_0$ is the radius of $\Delta$.
 In particular there exists $K_1>0$ such that
${K_1}^{-1}\leq \rho_{\omega,n}\leq K_1.$
Moreover,
$$\left|{\rho_{\omega,n}(x)-\rho_{\omega,n}(y)}\right|\leq||\rho_{\omega,n}
||_\infty\left|\frac{\rho_{\omega,n}(x)}{\rho_{\omega,n}(y)}-1\right|\leq
K_1C\left| \log\frac{\rho_{\omega,n}(x)}{\rho_{\omega,n}(y)} \right|\leq
C'd(x,y).$$
 By
Ascoli-Arzela theorem, the sequence $(\rho_{\omega,n})_n$ is relatively
compact
in $L^\infty(\Delta,m_0)$ and has some subsequence
$(\rho_{\omega,n_i^\omega})_i$ converging to some $\rho_{\omega}$.
By the construction process we have ${K_1}^{-1}\leq \rho_{\omega}\leq K_1$
and $\rho_{\omega}\in\mathcal{H}$. The measure $\nu_\w=\rho_\w
d\leb$ is finite  since $\nu_w(\Delta)\leq K_1m(\Delta)<\infty$. By a diagonalization argument we can now choose a suitable family $\{\nu_{\sigma^l(\w)}\}_{l\in\ZZ}$
of such finite measures satisfying the {\it quasi-invariance} property
${(F)_{\omega}}^{*}\{\nu_{\sigma^{-n}(\w)}\}_{n\in\NN}=\nu_{\omega}$.
\end{proof}

For each $\omega\in\Omega_\ep$,  the Lipschitz constant $K_{\rho_{\omega}}$ for $\rho_\w\in\mathcal{H}$
will
depend only on $K$ and $\kappa$
given by Theorem
\ref{random Markov structure}. By $(U2)$, considering
$\epsilon$ small
enough, the constants $K_{\rho_{\omega}}$ can be taken the same
for all $\w\in\Omega_\ep$, which we will refer as $K_2>0$.

 We define the family  $\{\tilde\mu_\w\}_{\w\in\Omega_\ep}$ of finite Borel measures on $M$ by
\begin{equation}\label{mu w nao normalizada}
{\tilde\mu}_{\omega}=\sum_{j=0}^{+\infty}
(f_{{\sigma}^{-j}(\omega)}^j)_*(\nu_{\sigma^{-j}(\omega)}\vert\{R_{{\sigma}^{-j}
(\omega)}>j\}),\end{equation}
where the measures $\nu_{\sigma^{-j}(\omega)}$ are given by Theorem
\ref{exist.mu.induced}.
Since
$$\tilde\mu_\w(M)=\sum_{j=0}^{+\infty}
\nu_{\sigma^{-j}(\omega)}(\{R_{{\sigma}^{-j}
(\omega)}>j\})\leq K_1 \sum_{j=0}^{+\infty}
\Leb(\{R_{{\sigma}^{-j}
(\omega)}>j\}),$$
the hypothesis
on the decay of $\Leb(\Gamma_\omega^n)$ and Theorem \ref{random Markov
structure} give that they are finite measures.
The absolute continuity of the measures $\{\nu_{w}\}_{\w\in\Omega_\ep}$ implies that the measures
of the family $\{\mu_{\omega}\}_{\w\in\Omega_\ep}$ are absolutely continuous and the quasi-invariance property for
$\{\nu_{w}\}_{\w\in\Omega_\ep}$ implies that ${f_\w}_*\tilde\mu_\w=\tilde\mu_{\sigma(\w)}$.

\cre\label{pastdependence}
By construction, all the measures in the family
$\{\nu_{\sigma^{-n}(\omega)}\}_{n\in\NN}$ depend only
in the past
$\w^-=(\ldots,\omega_{-2},\omega_{-1})$ of $\omega$. Moreover, for $\w,\tau\in T^\ZZ$ with the same past the
sets $\{R_{{\sigma}^{-j}(\omega)}=j\}$ and $\{R_{{\sigma}^{-j}(\tau)}=j\}$, for
$j\geq1$, are exactly the same (as subsets of $\Delta\subset M$).
The
measures $\tilde{\mu}_\w$ involve sums of the type
$(f_{{\sigma}^{-j}(\omega)}^j)_*(\nu_{\sigma^{-j}(\omega)}\vert\{R_{{\sigma}^{-j
}(\omega)}>j\})$,
and since
$$m(\{R_{{\sigma}^{-j}(\omega)}>j\})=m\left(\Delta\setminus\left\{
\displaystyle\bigcup_{k=1}^{j}\{R_{{\sigma}^{-k}(\omega)}=k\}\right\}
\right)$$ and
$\nu_{\sigma^{-j}(\omega)}\ll m$, the measures $\tilde\mu_\w$ depend only on the past $\w^-$ of $\w$.
\fre

\cle  $\displaystyle\tilde\mu_\ep=\int\tilde\mu_\w \,d\theta_\ep^\ZZ(\w)$ is an  absolutely continuous stationary finite measure.
\fle
\begin{proof}
Since $\{\tilde\mu_\w\}_{\w\in\Omega_\ep}$ almost surely depend only on the past,
then $\tilde\mu_\ep=\int\tilde\mu_\w \,d\theta_\ep^\ZZ(\w)$ is a stationary measure. Actually,
for every continuous map $\varphi:M\to\RR$ we have
\begin{eqnarray*}
 \int \varphi(x) \,d\tilde\mu_\ep(x) &=&\iint 	 \varphi (x) \,d\tilde\mu_\w(x)
	d\theta_\ep^\ZZ(\w)\\
  &=& \iint \varphi (x)
 	\,d\tilde\mu_{\sigma(\w)}(x) d\theta_\ep^\ZZ(\w)
\\
&=& \iiiint (\varphi\circ f_{\w_0})(x)
 	 \,d\tilde\mu_{\w}(x)d\theta_\ep^{\ZZ^-}(\w^-)
	 d\theta_\ep(\w_0)d\theta_\ep^{\ZZ^+}(\sigma^+(\w^+))\\
&=& \iint (\varphi\circ f_{\w_0})(x)
 	\,d\tilde\mu_\ep(x)
	d\theta_\ep(\w_0).
\end{eqnarray*}
Moreover, $\tilde\mu_\ep$ is absolutely continuous due to the absolute
continuously of
the measures $\{\tilde\mu_\w\}_{\w\in\Omega_\ep}$.
For the finiteness of $\tilde\mu_\ep$ we have we have by Theorem~\ref{random Markov structure} that
$$\tilde\mu_\ep(M)=\int\tilde\mu_\w(M)\,d\theta_\ep^\ZZ(\w)=\int\sum_{
n=0}^{+\infty}\nu_{\sigma^{-n}(\w)}(\{R_{\sigma^{-n}(\w)}>n\})\,d\theta_\ep^\ZZ(\w)\leq K_1\sum_{
n=0}^{+\infty}C'n^{-p}<\infty.$$
\end{proof}

 We now normalize $\tilde\mu_\ep$ and define an
absolutely continuous stationary probability measure
$\mu_\ep=\tilde\mu_\ep/\tilde\mu_\ep(M)$. Next we  prove that
  $\mu_\ep$ is the
unique absolutely continuous ergodic stationary probability measure.

\subsection{Ergodicity and uniqueness}\label{eau}

We say that $A\subset M$ is a {\em random forward invariant} set if for
$\mu_\ep$ almost every
$x\in A$ we have $f_t(x)\in A$ for  $\theta_\ep$ almost every $t$.
Let $\delta_1>0$ be given by Lemma
\ref{l.contr}.

\cpr \label{l.disco} Given any random forward
invariant set
$A\subset M$  with $\mu_\ep(A)>0$, there is a ball $B$ of radius $\delta_1/4$
such
that $\leb(B\setminus A)=0$. \fpr

\dem
It is enough to
prove that there exist disks of radius $\delta_1/4$ where the
relative measure of $A$ is arbitrarily close to one.
For $n\geq1$ let $A_n$ be the set of points $x\in A$ for which $f_\w^n(x)\in
A$, for $\theta_\ep^\ZZ$ almost every $\w$, and
$\tilde A=\displaystyle\cap_{n=1}^{+\infty}A_n$. Since $A$ is
random forward
invariant then
$\mu_\ep(A\setminus A_n)=0$ for all $n$ and thus
$\mu_\ep(A\setminus\tilde A)=0$. We have $\mu_\ep(\tilde
A)=\mu_\ep(A)>0$ which implies $\leb(\tilde A)>0$, since $\mu_\ep\ll\leb$. Morever, for
$\theta_\ep^\ZZ$ almost
every $\w$ we have $f^k_\w(\tilde A)\subset
A$, for all $k\in\NN$.
Since the set of points $x\in M$ for which for $\theta_\ep^\ZZ$ almost every
$\w$ there are infinitely many hyperbolic
pre-balls $V_\w^n(x)$ is random forward invariant, we
may assume, with no loss of generality,
that every point in $A$ has infinitely many hyperbolic pre-balls $V_\w^n(x)$
for $\theta_\ep^\ZZ$ almost every $\w$.  Recall that the hyperbolic pre-balls $V_\w^n(x)$ are sent diffeomorphically
by $f_\w^n$ onto hyperbolic balls  with radius $\delta_1$, that is
$f_\w^n(V_\w^n(x))=B(f_\w^n(x),{\delta_1})$.

Let
$\gamma>0$ be some small number. By regularity of $m$, there is a compact set
$\tilde A_c\subset \tilde A$
and  an open set $\tilde A_o\supset \tilde A$ such that
\begin{equation}\label{labep2}
\leb(\tilde A_o\setminus \tilde A_c)<\gamma \Leb(\tilde A).
\end{equation}
Assume that $n_0$ is large
enough so that, for every $x\in\tilde A_c$ and $\theta_\ep^\ZZ$ almost every
$\w$, any hyperbolic
 preball $V_\w^n(x)$,
associated with the $(\lambda,\delta)$-hyperbolic time $n$ for $(\w,x)$, with
$n\ge
n_0$, is contained in $\tilde A_o$. Let $W_\w^{n}(x)$ be the part of $
V_\w^n(x)$
which is sent diffeomorphically by $f_\w^n$ onto the ball
$B(f_\w^{n}(x),{\delta_1/4})$. By compactness there are $x_1,\ldots,x_r\in
\tilde A_c$ and $n(x_1),\dots,n(x_r)\ge n_0$ such that
\begin{equation}\label{eq.ac}
\tilde A_c\subset W_\w^{n(x_1)}(x_1)\cup\ldots\cup W_\w^{n(x_r)}(x_r).
\end{equation}
For the sake of notational simplicity, for each $1\le i\le r$ we shall write
 $$V_\w^i=V_\w^{n(x_i)}(x_i),\quad W_\w^i=W_\w^{n(x_i)}(x_i)\qand n_i=n(x_i). $$
Assume that
\begin{equation*}
    \{n_1,\ldots,n_r\}=\{n_1^*,\ldots,n_s^*\},
\quad\text{with $n_1^*<n_2^*<\ldots<n_s^*$}.
\end{equation*}
Let $I_1\subset\mathbb{N}$ be a maximal subset of $\{1,\ldots,r\}$ such
that for each  $i\in I_1$ both  $n_i=n_1^*$, and $W_\w^i\cap
W_\w^j=\emptyset$ for every $j \in I_1$ with $j\ne i$. Inductively, we
define $I_k$  for $2\le k \le s$ as follows: supposing that
$I_1,\dots, I_{k-1}$ have already been defined, let $I_k$ be a
maximal set of $\{1,\dots,r\}$ such that for each $i\in I_k$ both
$n_i=n_k^*$, and 
$W_\w^i\cap W_\w^j=\emptyset$ for every
$j\in I_1\cup\ldots\cup I_{k}$ with $i\neq j$.

Define  $I=I_1\cup\cdots\cup I_s$. By construction we have that
$\{W_\w^i\}_{i\in I}$
is a family of pairwise disjoint sets. We claim that $\{V_\w^i\}_{i\in I}$
is a covering of~$\tilde A_c$.
To see this, recall that by construction, given any $W_\w^j$ with
$1\le j\le r$, there is some $i\in I$ with
$n(x_i)\le n(x_j)$
such that $W_\w^{x_j}\cap W_\w^{x_i}\ne\emptyset$. Taking images by
$f_\w^{n(x_i)}$ we have
$$
f_\w^{n(x_i)}(W_j)\cap B(f_\w^{n(x_i)}(x_i),{\delta_1/4})\ne\emptyset.
$$
It follows from Proposition~\ref{p.contr}  that
$$
\diam(f_\w^{n(x_i)}(W_\w^j))\le\frac{\delta_1}{2}
\lambda^{(n(x_j)-n(x_i))/2}\le\frac{\delta_1}{2},
$$
 and so
$$f_\w^{n(x_i)}(W_\w^j)\subset B(f_\w^{n(x_i)}(x_i),{\delta_1}).$$ This
gives that $W_\w^j\subset V_\w^i$.  We have proved that given any $W_\w^j$ with
$1\le j\le r$, there is $i\in I$ so that $W_\w^j\subset V_\w^i$. Taking into
account \eqref{eq.ac}, this means that $\{V_\w^i\}_{i\in I}$
is a covering of $\tilde A_c$.

 By Corollary~\ref{co.distortion2} one may find $\tau>0$ such that
 $$
 m(W_\w^i)\ge\tau m(V_\w^i), \quad\text{for all $i\in I$}.$$
Hence,
 \begin{eqnarray*}
    m\left(\bigcup_{i\in I}W_\w^i\right) &=&\sum_{i\in I} m(W_\w^i) \\
     &\ge& \tau \sum_{i\in I} m(V_\w^i)  \\
     &\ge&\tau m\left(\bigcup_{i\in I}V_\w^i\right)\\
     &\ge&\tau m(\tilde A_c).
 \end{eqnarray*}
From \eqref{labep2} one easily deduces that
$m(\tilde A_c)>(1-\gamma)m(\tilde A)$. Noting that the constant~$\tau$ does not
depend
on $\gamma$, choosing $\gamma>0$ small enough we may have
 \begin{equation}\label{eqep2}
    m\left(\bigcup_{i\in I}W_\w^i\right)>\frac\tau2m(\tilde A).
 \end{equation}
We are going to prove that
 \begin{equation}\label{eqabs2}
 \frac{m(W_\w^i\setminus\tilde A)}{m(W_\w^i)}< \frac{2\gamma}{\tau}
 ,\quad\text{for some $i\in I$}.
 \end{equation}
 This is enough for our purpose. First, since $f_\w^{n(x_i)}(\tilde A)\subset A$
and
$f_\w^{n(x_i)}$ is injective on $W_\w^i$ we have
  \begin{eqnarray*}
   \leb(f_\w^{n(x_i)}(W_\w^i)\setminus
A)&\leq&\leb(f_\w^{n(x_i)}(W_\w^i)\setminus f_\w^{n(x_i)}(\tilde A)) \\
 &=&\leb(f_\w^{n(x_i)}(W_\w^i\setminus\tilde A)).
   \end{eqnarray*}
 Therefore,
by Corollary~\ref{co.distortion2}, taking
$B=f_\w^{n(x_i)}(W_\w^i)$ as a ball of radius $\delta_1/4$ we have
  $$
  \frac{\leb(B\setminus
A)}{m(B)}\le\frac{m(f^{n(x_i)}(
  W_\w^i\setminus\tilde A))}{m(f_\w^{n(x_i)}(W_\w^i))}\le C_2
  \frac{m( W_\w^i\setminus\tilde A)}{m( W_\w^i
  )}=\frac{2C_2\gamma}{\tau},
  $$
  which can obviously be made arbitrarily small, setting $\gamma\to 0$. From
this one easily deduces that there are disks of radius $\delta_1/4$ where the
relative measure of $A$ is arbitrarily close to one.

Finally, let us prove~\eqref{eqabs2}. Assume, by contradiction, that it
  does not hold.  Then, using \eqref{labep2} and \eqref{eqep2}
  \begin{eqnarray*}
  \gamma m(\tilde A) &> & m(\tilde  A_o\setminus \tilde A_c) \\
   &\ge & m\left(\left(\bigcup_{i\in I}W_\w^i\right)\setminus \tilde A \right)\\
    &\ge &\frac{2\gamma}{\tau} m\left(\bigcup_{i\in I}W_\w^i\right) \\
     &>&\gamma m(\tilde A).
 \end{eqnarray*}
This gives a contradiction.\cqd

\begin{Proposition}\label{ergouniq}
The stationary measure $\mu_\ep$ is the unique absolutely continuous ergodic
stationary probability measure.
\end{Proposition}

\begin{proof}
We prove first that there is a finite partition $H_1,\ldots H_n$ of a full
Lebesgue measure set in
$M$ such that the normalized restrictions of $\mu_\ep$ to each $H_i$, $i =
1,\ldots,n$ is ergodic. Then, we use the topological transitivity of $f$ to
ensure the unicity.

 Suppose $\mu_\ep$ is is not ergodic. Then we may decompose $M$ into
two disjoint random invariant sets $H_1$ and $H_2$
($=M\setminus {H_1}$) both with positive
$\mu_\ep$-measure. In particular, both $H_1$ and $H_2$ have positive
Lebesgue measure. Let $\mu_\ep^1 $ and $ \mu_\ep^2$ be the normalized
restrictions of $\mu_\ep$ to $H_1$ and $ H_2$, respectively.
They are also absolutely continuous stationary measures. If they are
not ergodic, we continue decomposing them, in the same way as we did
for $\mu_\ep$.

On the other hand, by Proposition~\ref{l.disco}, each one of the random invariant
sets we find in this decomposition has full Lebesgue measure in some
disk with fixed radius. Since these disks must be disjoint, and $M$ is compact,
there can only be finitely many of
them. So, the decomposition must stop after a finite number of
steps, giving that $\mu_\ep$ can be written
$\mu_\ep=\sum_{i=1}^{p}\mu_\ep(H_i) \mu_\ep^i$ where $H_1, \dots, H_p$ is a
partition of $M$ into random invariant sets with
positive measure and each
$\mu_\ep^i=(\mu_\ep\vert H_i)/\mu_\ep(H_i)$ is an ergodic stationary probability
measure.

For the unicity assume that there are two distinct  ergodic
absolutely
continuous invariant measures $\mu_\ep^1$ and $\mu_\ep^2$. Since $B(\mu_\ep^1)$
and $B(\mu_\ep^2)$ are random forward invariant
sets, then by
Proposition~\ref{l.disco}  there are disks
$\Delta_1=B(p_1,{\delta_1/4})$ and $\Delta_2=B(p_2,{\delta_1/4})$ such that
$m(\Delta_i\setminus
B(\mu_\ep^i))=0$ for $i=1,2$. 
The topological transitivity of $f$, the continuity of $\Phi$ and the
random invariance
of $B(\mu_1)$ and
$B(\mu_2)$ imply that, if $\ep$ is small enough, then $m(B(\mu_\ep^1)\cap
B(\mu_\ep^2))>0$. Consider any point $x$ in this intersection. For
every continuous
function $\varphi\colon M\to\RR$ and a $\theta_\ep^\ZZ$ full subset of $T^\ZZ$,
$\frac1n\sum_{j=0}^{n-1}\varphi(f_\w^j(x))$ converges to $\int\varphi
\,d\mu_\ep^1$, as $n$ goes to infinity and, similarly, to $\int\varphi
\,d\mu_\ep^2$. The unicity of the limit implies $\int\varphi
\,d\mu_\ep^1=\int\varphi
\,d\mu_\ep^2$ so that $\mu_\ep^1=\mu_\ep^2$.

If we consider any stationary probability measure $\hat\mu_\ep$ on $M$, we can
do the
same procedure as before, and get a finite decomposition of $\hat\mu_\ep$ in
ergodic components, containing (Lebesgue mod 0) disks of a fixed radius. As we
saw, by the topological transitivity of $f$ one should have
$\hat\mu_\ep=\mu_\ep$.
\end{proof}

This finishes the proof of the first item of Theorem \ref{thA}.

\section{Strong stochastic stability}\label{se.stability}
In this section we prove the second item of Theorem~\ref{thA}.
We need to show the convergence of the density of the
unique absolutely continuous ergodic stationary probability measure $\mu_\ep$ for
$\{\Phi,(\theta_\ep)_{\ep>0}\}$ to the density of the unique $f$-invariant
absolutely continuous probability measure $\mu_f$, in the $L^1$-norm. The strategy is to get an absolutely continuous Borel measure $\nu_\ep$ on $\Delta$, with density $\rho_\ep$,
 by averaging over the previously constructed
induced random measures $\nu_\w$ on $\Delta$. The family of densities $(\rho_\ep)_{\ep>0}$ has an
accumulation point $\rho_\infty$ in $L^1(\Delta)$, as $\ep$ goes to $0$, which
give us a measure $\nu_\infty$ on $\Delta$. We can project this measure to a probability measure $\mu_\infty$ on $M$, using the dynamics of $f$,
in such a way that we can compare the densities of the measures $\mu_\ep$ and $\mu_\infty$ on
$M$, attesting their convergence in the $L^1$-norm. To conclude we  prove in Proposition \ref{mu infty is f invariant}
that
$\mu_\infty$ is $f$-invariant and
is actually equal
to $\mu_f$ due to the unicity of the SRB measure.

We start by defining an absolutely continuous measure $\nu_\ep$ on $\Delta$, with density
$\rho_\ep$, as
$$\nu_\epsilon =\int\nu_{\omega} \,d\theta^{\ZZ}_\epsilon(\omega)
\quad\text{and}\quad\rho_\ep =\int
\rho_\omega \,d\theta_\ep^\ZZ(\omega).$$
Consider any sequence $(\epsilon_n)_n$, with $\ep_n>0$ and
$\ep_n\to0$, as $n\to+\infty$. The family $\{\rho_{\ep_n}\}_n$ is
uniformly
bounded and equicontinuous. Indeed,
${K_1}^{-1}\leq\rho_{\ep_n}\leq K_1$
and, for $x,y\in\Delta$ we have
$|\rho_{\ep_n}(x)-\rho_{\ep_n}(y)|\leq K_2 d(x,y)$.
By Ascoli-Arzela theorem this implies that $({\rho_\ep}_n)_n$ has a converging
subsequence $\{\rho_{\ep_n'}\}$ to some $\rho_\infty$ in the $L^1$-norm, with
${K_1}^{-1}\leq\rho_\infty\leq K_1$.
Without loss of
generality we will assume that the whole sequence converges; see Remark~\ref{unicity}.
 This means
that given $\gamma>0$ there exists $\delta>0$ such that, if
$\ep_n<\delta$, then
\begin{equation}\label{convergencia induzidas}
\left||\rho_{\ep_n}-\rho_\infty\right\|_1<\gamma.
\end{equation}
Let $\nu_\infty$   be the Borel measure on $\Delta$ with density
$\rho_\infty$ and define an absolutely continuous Borel measure
$\tilde\mu_\infty$ in $M$ by
 \begin{equation}
\tilde\mu_{\infty}=\sum_{j=0}^\infty
(f^j)_*(\nu_{\infty}\vert\{R>j\}).\end{equation}
Let
$\tilde{h}_\infty=d\tilde\mu_{\infty}/dm$.
Since  $f$ is non-uniformly expanding, the decay of
the return time together with the upper bound for $\rho_\infty$
imply that $\tilde\mu_\infty$ is finite.
We normalize it to obtain an absolutely continuous probability measure
$\mu_\infty$
with
density $h_\infty$.

We introduce the transfer operator acting on $L^1(\Delta)$ by
\begin{equation*}\label{operadorLt}{\mathcal{L}}_{\omega}^j\varphi(x)=\sum_{
y=(f_{\omega}^{j})^{-1}(x)}\frac{\vphi(y)}{|\det Df_{\omega}^j(y)|}
\end{equation*}
and, to simplify the notation, we write
${\mathcal{L}}^j\varphi(x)$ for ${\mathcal{L}}_{
\w^*}^j\varphi(x)$. Standard results on transfer operators give that $$\int({\mathcal{L}}_{\omega}^j\varphi)\psi \,d\leb=\int\varphi(\psi\circ
f_{\omega}^j)\,d\leb$$
whenever these integrals
make sense, and the operator does not expand:
$$\int\left|\mathcal{L}_{\omega}^j\varphi(x)\right|\,d\leb\leq\int\mathcal{L}_{\omega}^j\left|\varphi(x)\right|\,d\leb =\int|\vphi|\, d\leb$$
for every $\vphi\in L^1(\Delta)$.

\cle\label{desOp2}   Given $\gamma>0$ and $j\in\NN$ there is
$\alpha_j>0$ such that if $\ep<\alpha_j$ then for every
$\omega,
\tau\in\Omega_\epsilon$ we have
$$\int\left|\!\mathcal{L}_{\omega}^j\varphi(x)-\mathcal{L}_{\tau}
^j\varphi(x)\right|\, d\leb \leq \gamma\|\f\|_\infty$$
for every $\f$ in $\mathcal{H}$. \fle

\dem
Let $\f\in\mathcal{H}$. Our assumptions on the critical set imply that the critical set $\cc$
intersects
$\Delta$  in a zero Lebesgue measure set. Given any $\gamma_1>0$,
define $\cc(\gamma_1)$ as the $\gamma_1$-neighborhood of this
intersection. For every $\omega\in\Omega_\epsilon$, we have
$m(f_{\omega}^j(\cc(\gamma_1)))\le\const
m(\cc(\gamma_1))$ for
some constant that may be taken uniform over $\w$ if $\epsilon<\alpha_j$ for
some small $\alpha_j$.

 By the continuity of $f^j_{\omega}$ with respect to $\omega$, we may fix
$\gamma_1=\gamma_1(j)$ small enough so
that if $\epsilon<\alpha_j<\gamma_1$ then
 \begin{equation}
 \label{eq.fi1fi2}
 m\big((f_{\omega}^{j})^{-1}(f_{\tau}^j (\cc(\gamma_1))\big)
  \le \frac{1}{2}\left(\frac{\gamma}{8}\right),
 \end{equation}
for every ${\omega}$, $\tau$ in $\Omega_\epsilon$.

We
decompose $\Delta\setminus\cc(\gamma_1)$ into a finite collection
$\cad(\omega)$ of domains of injectivity of $f_{\omega}^j$.
We may define a
corresponding collection $\cad(\tau)$ of domains of injectivity for
$f_{\tau}^j$ in $\Delta\setminus\cc(\gamma_1)$, and there is a natural
bijection associating to each $D_{\omega}\in\cad(\omega)$ a unique
$D_{\tau}\in\cad(\tau)$ such that the Lebesgue measure of $D_{\tau}\triangle
D_{\omega}$ is small, where $D_{\tau}\triangle D_{\omega}$  denotes the
symmetric
difference of the two sets $D_{\tau}$ and $D_{\omega}$. Observe that
$\cl_{\tau}^j$
is supported in
 $$
 f_{\tau}^j(\Delta)=f_{\tau}^j(\cc(\gamma_1)) \cup
\bigcup_{D_{\tau}\in\cad(\tau)} f_{\tau}^j(D_{\tau}),
 $$
and analogously for $\cl_{\omega}^j$\,.
 So,
 \begin{eqnarray}
\int |\cl_{\tau}^j \varphi -\cl_{\omega}^j \varphi |\,d\leb &\leq &
\int_{f_{\omega}^j(\cc(\gamma_1))\cup
f_{\tau}^j(\cc(\gamma_1)) }
 \big(|\cl_{\tau}^j \f | +|\cl_{\omega}^j \f|\big)\,d\leb \label{e1}
 \\ & + &
\sum_{D_{\omega}\in\cad (\omega)}\int_{f_{\omega}^j(D_{\omega})\cap
f_{\tau}^j(D_{\tau})}|\cl_{\tau}^j
\f-\cl_{\omega}^j\f|\,d\leb\label{e2}
 \\ & + &
 \sum_{D_{\omega}\in\cad(\omega)}
 \int_{f_{\omega}^j(D_{\omega})\triangle f_{\tau}^j(D_{\tau})}
 \big(|\cl_{\tau}^j \f| + |\cl_{\omega}^j \f|\big)\,d\leb, \label{e34} \end{eqnarray}
 where $D_{\tau}$ always denotes the element of $\cad(\tau)$ associated to each
$D_{\omega}\in\cad(\omega)$. Let us now estimate the expressions on the
right hand side of this inequality. We start with (\ref{e1}). For
notational simplicity, we write $E=f_{\omega}^j(\cc(\gamma_1))\cup
f_{\tau}^j(\cc(\gamma_1))$. Then
 \begin{equation*}
 \int_{E}|\cl_{\tau}^j\f|\,d\leb
 \leq  \int\chi_{E}\big(\cl_{\tau}^j|\f|\big)\,d\leb
 = \int(\chi_{E}\circ f_{\tau}^j)|\f|\,d\leb.
 \end{equation*}
 It follows from Hölder's inequality and \eqref{eq.fi1fi2} that

 $$
\int(\chi_{E}\circ f_{\tau}^j)|\f|\,d\leb
 \leq m\big((f^{j}_{\tau})^{-1}E\big)\|\f\|_\infty
 \leq \frac{\gamma}{8} \|\f\|_\infty.
 $$

The case associated to $f_{\omega}^j$ gives a similar bound for the second term in
(\ref{e1}). So,
 \begin{equation}
  \int_{f_{\omega}^j(\cc(\gamma_1))\cup
f_{\tau}^j(\cc(\gamma_1))}\big(|\cl_{\tau}^j\f| + |\cl_{\omega}^j \f|\big) \,d\leb
  \leq \frac{\gamma}{4}\|\f\|_\infty.
  \label{f1}
 \end{equation}
Making the change of variables $y=f_{\omega}^j(x)$ in (\ref{e2}), we may
rewrite it as
 $$
\int_{\widehat{D_{\omega}}}\left|\frac{\f}{|\det
Df_{\tau}^j|}\circ((f_{\tau}^{j})^{-1}\circ
f_{\omega}^j) -\frac{\f}{|\det Df_{\omega}^j|}\right|\cdot|\det Df_{\omega}^j|\,d\leb,
 $$
where
 $
\widehat{D_{\omega}}=(f_{\omega}^{j})^{-1}\big(f_{\omega}^j(D_{\omega})\cap
f_{\tau}^j(D_{\tau})\big)=
D_{\omega}\cap\big((f_{\omega}^{j})^{-1}\circ
f_{\tau}^j\big)(D_{\tau}).
 $
 For notational simplicity, we introduce
 $g = (f_{\tau}^{j})^{-1}\circ f_{\omega}^j$.
 The previous expression is bounded by
 $$
 \int_{\widehat
D_{\omega}} \left(|\f\circ g-\f|\cdot \frac{|\det Df_{\omega}^j|}{|\det
Df_{\tau}^j|\circ
g}+|\f|\cdot \left|\frac{|\det Df_{\omega}^j|}{|\det Df_{\tau}^j|\circ
g}-1\right|
\right)\,d\leb.
 $$
 Choosing $\alpha_j>0$ sufficiently small, the assumption
  $\epsilon<\alpha_j$ implies
 $$
\left|\frac{|\det Df_{\omega}^j|}{|\det Df_{\tau}^j|\circ
g}-1\right|\leq\frac{\gamma}{8},
 \quad\mbox{and so}\quad
 \frac{|\det Df_{\omega}^j|}{|\det Df_{\tau}^j|\circ g}\leq 2
 $$
 on
$\Delta\setminus\cc(\gamma_1)$ (which contains $\widehat
D_{\omega}$). Hence, since $\f$ belongs to $\mathcal{H}$,
 \begin{eqnarray*}
 \int_{f_{\omega}^j(D_{\omega})\cap f_{\tau}^j(D_{\tau})}|\cl_{\tau}^j
\f-\cl_{\omega}^j\f|\,d\leb
 & \leq & 2 \int_{\widehat D_{\omega}} |\f\circ g-\f|\,d\leb
+\frac{\gamma }{8}\int|\f|\,d\leb
  \\
 & \leq & 2 K_\f  \int_{\widehat D_{\omega}}
\|g-\id_\Delta\|_0\,d\leb
+\frac{\gamma}8\|\f\|_\infty
   \end{eqnarray*}
 Reducing $\alpha_j>0$, we can make $\|g-\id_\Delta\|_0$
 smaller than $\frac{\gamma\|\f\|_\infty}{16K_\f m(\Delta)}$, so that

  \begin{equation}
 \int_{f_{\omega}^j(D_{\omega})\cap f_{\tau}^j(D_{\tau})}|\cl_{\tau}^j
\f-\cl_{\omega}^j\f|\,d\leb \leq \frac{\gamma}{4}\|\f\|_\infty
.\label{f2}
 \end{equation}
We estimate the terms in (\ref{e34}) in much the same way as we
did for (\ref{e1}). For each $D_{\omega}$ let $E$ now be
$f_{\omega}^j(D_{\omega})\triangle f_{\tau}^j(D_{\tau})$. The properties of
the
operator, followed by Hölder's inequality, yield
 \begin{eqnarray*}
\int_{E}|\cl^j_{\tau} \f|\,d\leb
 & \leq &
\int(\chi_{E}\circ f_{\tau}^j)|\f|\,d\leb
 \leq m\big((f^{j}_{\tau})^{-1}E\big)\|\f\|_\infty
  \end{eqnarray*}
 Fix $\gamma_2>0$ such that $\#\cad(\omega)4\gamma_2 < \gamma$.
 Taking $\alpha_j$ sufficiently small, we may ensure that the
 Lebesgue measure of all the sets
 $$
 (f^{j}_{\tau})^{-1}E = (f^{j}_{\tau})^{-1}(f_{\omega}^j(D_{\omega})\triangle
f^j_{\tau}(D_{\tau}))
 $$
 is small enough so that $$m\big((f^{j}_{\tau})^{-1}E\big)<\gamma_2.$$
In this way we get
 \begin{equation}
\int_{f_{\omega}^j(D_{\omega})\triangle f_{\tau}^j(D_{\tau})}
 \big(|\cl_{\omega}^j \f| + |\cl_{\tau}^j\f|\big) \,d\leb
  \leq 2\gamma_2\|\f\|_\infty\label{f3}
 \end{equation}
 (the second term on the left is estimated in the same way as the first
 one).
  Putting (\ref{f1}), (\ref{f2}), (\ref{f3}) together, we obtain
 $$
 \int |\cl_{\tau}^j\f-\cl_{\omega}^j\f |\,d\leb
  \leq
\left(\frac{\gamma}{2}+\#\cad(\omega)2\gamma_2\right)\|\f\|_\infty<\gamma\|\f\|_\infty $$
which concludes the proof. \cqd

\cpr\label{estabilidadeestocastica}
Let $(\ep_n)_n$ be a sequence such that $\ep_n>0$ and
$\ep_n\to 0$, as ${n\to\infty}$. The
density $h_{\ep_n}$ converges to $h_\infty$ in the $L^{1}$-norm.\fpr

\dem
For simplicity we prove that $\|\tilde
h_{\ep_n}-\tilde
h_\infty\|_{1}$ converges to zero as $\ep_n$ goes to zero, where $\tilde h_{\ep_n}=d\tilde\mu_{\ep_n}/d\leb$, which implies the desired result.
For given
$\gamma>0$ we are looking for $\alpha>0$
such that if $\epsilon_n<\alpha$ then
$\| \tilde h_{\ep_n}- \tilde h_\infty\|_{1}<\gamma$.
By \eqref{uniform tail decay} there is an integer $N\geq 1$
for which \begin{equation}\label{usar uniform tail decay}
\sum_{j=N+1}^{{+\infty}}\Leb(\{R_{\sigma^{-j}(\omega)}>j\})<\frac{
\gamma}{4K_1},\quad \forall\omega\in\Omega_{\epsilon_n}.\end{equation}
We split the following sums
\begin{equation*}
\tilde\mu_{\infty}=\sum_{j=0}^N\xi_{\infty,j}\,\,+\eta_{\infty,N}\quad\text{and}
\quad
     \tilde
\mu_{\ep_n}=\sum_{j=0}^N\xi_{\ep_n,j}\,\,+\eta_{\infty,N},
\end{equation*}
   where, for every $j=0,1,2,\ldots,N$, we have
\begin{eqnarray*}\xi_{\infty,j}=(f^j)_*(\nu_{\infty}\vert\{R>j\}),& &
\xi_{\ep_n, j }=\int
(f^j_{{\sigma}^{-j}(\omega)})_*(\nu_{{\sigma}^{-j}(\omega)}\vert\{R_{{\sigma}^{
-j}(\omega)}>j\})\,d\theta_{\epsilon_n}^{\ZZ}(\omega)\end{eqnarray*}
and the remaining sums are
\begin{eqnarray*}
   \quad\eta_{\infty,N}&=&\sum_{j=N+1}^\infty
(f^j)_*(\nu_{\infty}\vert\{R>j\})\\
 \eta_{\ep_n,N}&=&\int\sum_{j=N+1}^{+\infty}(f^j_{{\sigma}^{-j}(\omega)})_*(\nu_
{{
\sigma}^{-j}(\omega)}\vert\{R_{{\sigma}^{-j}(\omega)}>j\})\,d\theta_{\epsilon_n}
^{\ZZ}(\omega).
\end{eqnarray*}
Recall that the realization $\w^*$, which reproduces the original deterministic
dynamical system, belongs to $\Omega_{\ep_n}$. By \eqref{usar uniform tail decay} we have

\vspace{-2cm}

\begin{eqnarray*}
\phantom{\displaystyle\eta_{\ep_n,N}(M)}
 &\phantom{=}&\phantom{
\int\sum_{j=N+1}^{+\infty}
(f^j_{{\sigma}^{-j}(\omega)})_*(\nu_{{\sigma}^{-j}(\omega)}\vert\{R_{{\sigma}^{
-j}(\omega)}>j\})(M)\,d\theta_{\ep_n}^{\ZZ}
}\\ \vspace{-18cm}
\displaystyle\eta_{\infty,N}(M)
 &=&
\sum_{j=N+1}^{+\infty} (f^j)_*(\nu_{\infty}\vert\{R>j\})(M)\\
 &=&\sum_{j=N+1}^{+\infty} \nu_{\infty}(\{R>j\})\\
 &=&\sum_{j=N+1}^{+\infty} \int \rho_{\infty}\chi_{\{R>j\}}\,d\leb\\
 &\leq&K_1\sum_{j=N+1}^{+\infty} \Leb(\{R>j\})\\
 &\leq&\frac{\gamma}4.
\end{eqnarray*}
Similarly,
\begin{eqnarray*}
\displaystyle\eta_{\ep_n,N}(M)
 &=&
\int\sum_{j=N+1}^{+\infty}
(f^j_{{\sigma}^{-j}(\omega)})_*(\nu_{{\sigma}^{-j}(\omega)}\vert\{R_{{\sigma}^{
-j}(\omega)}>j\})(M)\,d\theta_{\ep_n}^{\ZZ}\\
 &=&\int\sum_{j=N+1}^{+\infty}
\nu_{{\sigma}^{-j}(\omega)}(\{R_{{\sigma}^{-j}(\omega)}>j\})\,d\theta_{\ep_n}^{
\ZZ}\\
 &=&\int\sum_{j=N+1}^{+\infty} \int
\rho_{{\sigma}^{-j}(\omega)}\chi_{\{R_{{\sigma}^{-j}(\omega)}>j\}}\,d\leb\,d\theta_
{\ep_n}^{\ZZ}\\
  &\leq&K_1\int\sum_{j=N+1}^{+\infty}
\Leb(\{R_{{\sigma}^{-j}(\omega)}>j\})\,d\theta_{\ep_n}^{\ZZ}\\
 &\leq&\frac{\gamma}4.
\end{eqnarray*}

 Altogether this gives us
\begin{equation}\label{primeiroe2}
\left\|\frac{d\eta_{\infty,N}}{d\leb}-\frac{d\eta_{\ep_n,N}}{d\leb}\right\|_1
\leq \eta_{\infty,N}(M)+\eta_{\ep_n,N}(M)\leq\frac\gamma 2.
\end{equation}
By \eqref{convergencia induzidas} there is some $\alpha_0>0$ such that
$\ep_n<\alpha_0$ implies
\begin{equation}\label{A0}\left\|\frac{d\xi_{\ep_n,0}}{d\leb}-\frac{d\xi_{\infty,0}
} {d\leb}\right\|_1=\left\|\rho_ { \ep_n
}-\rho_\infty\right\|_1<\frac{\gamma}{4}.\end{equation} On
the other hand, for every $j=1,2,\ldots,N$
\begin{eqnarray*}\label{desnu}
\left\|\frac{d\xi_{\ep_n,j}}{d\leb}-\frac{d\xi_{\infty,j}}{d\leb}\right\|_1&=&
\left\|\int\mathcal{L}_{\sigma^{-j}(\omega)}^j(\rho_{\sigma^{-j}(\omega)}\chi_{\{R_{{\sigma}^{-j}(\omega)}>j\}})\,d\theta_{\ep_n}^{\ZZ}-
\mathcal{L}^j(\rho_{\infty}\chi_{\{R>j\}})\right\|_1\\
&\leq& A_j+B_j+C_j,
\end{eqnarray*}
where
\begin{eqnarray*}
A_j&=&\left\|\int \mathcal{L}_{\sigma^{-j}(\omega)}^j(\rho_{\sigma^{-j}(\omega)}\chi_{\{R_{{\sigma}^{-j}(\omega)}>j\}})
-\mathcal{L}^j(\rho_{\sigma^{-j}(\omega)}\chi_{\{R_{{\sigma}^{-j}(\omega)}>j\}})\,d\theta_{\ep_n}^{\ZZ}\right\|_1
\\
B_j&=&\left\|\int\mathcal{L}^j(\rho_{\sigma^{-j}(\omega)}\chi_{\{R_{{\sigma}^{-j}(\omega)}>j\}}
-\rho_{\sigma^{-j}(\omega)}\chi_{\{R>j\}})\,d\theta_{\ep_n}^{\ZZ}\right\|_1 \\
\\
C_j&=&\left\|\int\mathcal{L}^j(\rho_{\sigma^{-j}(\omega)}\chi_{\{R>j\}})\,d\theta_{\ep_n}^{\ZZ}
-\mathcal{L}^j(\rho_{\infty}\chi_{\{R>j\}})\right\|_1.
\end{eqnarray*}
The Lemma \ref{desOp2} implies that there exists some $\alpha_j>0$ such
that if $\ep_n<\alpha_j$ then
\begin{eqnarray*}\displaystyle
A_j&\leq&\iint\left|
\mathcal{L}_{\sigma^{-j}(\omega)}^j(\rho_{\sigma^{-j}(\omega)}\chi_{\{R_{{\sigma
}^{-j}(\omega)}>j\}})
-\mathcal{L}^j(\rho_{\sigma^{-j}(\omega)}\chi_{\{R_{{\sigma}^{-j}(\omega)}>j\}}
)\right|\,d\leb d\theta_{\ep_n}^{\ZZ}
\\
&\leq&\iint\left|
\mathcal{L}_{\sigma^{-j}(\omega)}^j(\rho_{\sigma^{-j}(\omega)})
-\mathcal{L}^j(\rho_{\sigma^{-j}(\omega)}
)\right|\,d\leb d\theta_{\ep_n}^{\ZZ}
\\&\leq&\frac\gamma{12N}.
\end{eqnarray*}
 We also consider $\alpha_j$ small enough so that, by condition $(U1)$, if $\ep_n<\alpha_j$ then
   \begin{equation*}
\Leb\left(\{R_{\sigma^{-j}(\w)}=l\}\triangle\{R=l\}\right)\leq\frac{
\gamma }{12jNK_1},
 \end{equation*}
for $l=1, 2,\ldots, j$ and $\w\in\Omega_{\ep_n}$. Since
\begin{eqnarray*}
 \Leb(\{R_{\sigma^{-j}(\w)}>j\}\triangle\{R>j\})
&=&\Leb\left(\left(\Delta\setminus\bigcup_{l=1}^j\{R_{\sigma^{-j}(\w)}
=l\}\right)
\triangle\left( \Delta\setminus\bigcup_{l=1}^j\{R=l\}
\right)\right)\\
&\leq&\sum_{l=1}^j\Leb\left(\{R_{\sigma^{-j}(\w)}=l\}\triangle\{R=l\}
\right)\\ &\leq&\sum_{l=1}^j\frac{
\gamma }{12jNK_1}=\frac{
\gamma }{12NK_1}
\end{eqnarray*}
we have then
\begin{eqnarray*}\label{B1 leq}
B_j&\leq&\iint\left|\mathcal{L}^j\left(\rho_{\sigma^{-j}(\omega)}(\chi_{\{R_{{
\sigma}^{-j}(\omega)}>j\}}-\chi_{\{R>j\}})
\right)\right|\,d\leb d\theta_{\ep_n}^{\ZZ}\\
&\leq&\iint\left|\rho_{\sigma^{-j}(\omega)}\right| \left|
\chi_{\{R_{{\sigma}^{-j}(\omega)}>j\}}-\chi_{\{R>j\}}\right|\,d\leb d\theta_{\ep_n}^{\ZZ}\\
&\leq&K_1\int\Leb\left(\{R_{\sigma^{-j}(\w)}>j\}\triangle\{R>j\}
\right)\,d\theta_{\ep_n
}^{\ZZ}\\
&\leq&\frac\gamma{12N}.\\
&\phantom{\leq}&\phantom{\int\sum_{y=f^{-j}(x)}\frac{1}{\left|\det
Df^j(y)\right|}\left|\int\rho_{\sigma^{-j}(\omega)}(y)-\rho_{\infty}
(y)\,d\theta_{\ep_n}^{\ZZ}\right|d\leb}
\end{eqnarray*}

\vspace{-1cm}

If $\alpha_0$ is small enough, for $\epsilon_n<\alpha_{0}$ we also have
\begin{eqnarray*}\label{B2 leq}
C_j&\leq&\left\|\int\mathcal{L}^j(\rho_{\sigma^{-j}(\omega)})\,d\theta_{\ep_n}^{\ZZ}
-\mathcal{L}^j(\rho_{\infty})\right\|_1\\
&=&\left\|\mathcal{L}^j(\rho_{\ep_n})
-\mathcal{L}^j(\rho_{\infty})\right\|_1\\
&\leq&\left\|\rho_{\ep_n}
-\rho_{\infty}\right\|_1\\
&\leq&\frac\gamma{12N}.
\end{eqnarray*}
Altogether, considering
$\epsilon_n<\min_{j\in\{0,1,\ldots,N\}}\,\{\alpha_j\}$
we get
$A_j+B_j+C_j<\frac\gamma{4N}$, for $1\leq j\leq N$. The sum over all these $j$'s added to the
superior limit at \eqref{A0} is less
than $\frac\gamma2$. The estimate \eqref{primeiroe2} completes the
proof. \cqd

\cpr\label{mu infty is f invariant} The measure $\mu_{\infty}$ is
$f$-invariant.\fpr

 \dem  Take any continuous map $\varphi:M\to \RR$. Since
\begin{equation*}
\left|\int\varphi\,d\mu_{\ep_n}-\int\varphi\,d\mu_{\infty}\right|\leq
\|\varphi\|_\infty\|  h_{\ep_n}-  h_\infty\|_1\end{equation*}
and $\|  h_{\ep_n}-  h_\infty\|_{1}\to 0$,
as $\ep_n\to 0$, then $\mu_{\ep_n}$ converges to
$\mu_\infty$ in weak$^*$ sense and
\begin{equation*}\int\varphi\,d\mu_{\ep_n}\to\int\varphi\,d\mu_{\infty}.\end{equation*}
However, since $\mu_{\ep_n}$ is a stationary measure we have
\begin{equation*} \int\varphi(x)\,d\mu_{\ep_n}(x)=\iint(\varphi\circ
f_{s})(x)\,d\mu_{\ep_n}(x)\,d\theta_{\ep_n}(s).\end{equation*}
It suffices then to prove that
$\displaystyle\iint(\varphi\circ
f_{s})\,d\mu_{\ep_n}\,d\theta_{\ep_n}\,\to\int(\varphi\circ f)\,d\mu_{\infty}.$
So,
\begin{eqnarray*}
&&\left|\iint (\varphi\circ
f_{s})\,d\mu_{\ep_n}d\theta_{\ep_n}\,-\int(\varphi\circ
f)\,d\mu_{\infty}\right|\\
&\leq&\left|\iint (\varphi\circ
f_{s})\,d\mu_{\ep_n} d\theta_{\ep_n}\,-\int(\varphi\circ
f)\,d\mu_{\ep_n}\right|+\left|\int(\varphi\circ
f)\,d\mu_{\ep_n}\,-\int(\varphi\circ f)\,d\mu_{\infty}\right|.
\end{eqnarray*}
For $\ep_n$ sufficiently small, $(\varphi\circ f_{s}-\varphi\circ f)$ is uniformly close
to $0$, for every $s\in\supp \theta_{\ep_n}$. The
second
term is smaller than $\left\|\varphi\right\|_\infty\left\|h_{\ep_n}-h_\infty\right\|_1$, which is
close to zero if $\ep_n
$ is small enough.
 \cqd

\cre \label{unicity} The unicity
of an SRB measure $\mu_f$ for $f$ ensures that in the previous arguments we can consider
all the sequence $\ep_n$ instead just a subsequence of it. To see
this, for every subsequence of $\ep_n$ we can repeat the previous process and obtain a new subsequence ${\ep_n'}$
for which the corresponding sequence of densities $(h_{\ep_n'})_{n}$ has limit $h_{\infty}'$, and $\mu_f'=h_{\infty}'dm$ is also a $f$-invariant SRB measure (thus equal to $\mu_f$).  On the other hand, one knows that if all subsequences of a given sequence admits a subsequence
converging to a same limit then the whole sequence converges to that limit.
\fre
This finishes the proof of Theorem \ref{thA}.

\section{Induced Gibbs-Markov maps}\label{se.monitoring}

In this section we prove Theorem~\ref{random Markov structure}.

\subsection{Hyperbolic times and bounded
distortion}

Hyperbolic times were introduced in \cite{A00} for deterministic
systems
and extended in \cite{AAr03} to a random context. Here we recall the definition and the main properties.
For the next definition we fix $B>1$ and
$\beta>0$ as in
Definition~\ref{d.nd}, and take a constant $b>0$ such that $2b <
\min\{1,\beta^{-1}\}$. \cde Given $0<\lambda<1$ and $\delta>0$, we
say that $n\in\NN$ is a $(\lambda,\delta)$-hyperbolic time for
$(\omega,x)\in T^\ZZ\times M$ if, for every $1\leq k\leq n$,
 \begin{equation}\label{d.htr}
 \prod_{j=n-k}^{n-1}
 \|Df_{\sigma^j(\w)}(f^j_{\omega}(x))^{-1}\|\leq\lambda^k
 \quad\mbox{and}\quad
 \dist_\delta(f_{\omega}^{n-k}(x),\cc)\geq \lambda^{bk}.
 \end{equation}
In the case of $\cc=\emptyset$ the definition of
$(\lambda,\delta)$-hyperbolic time reduces to the first condition
in \eqref{d.htr} and we simply call it a {\em $\lambda$-hyperbolic
time}.\fde

We define, for $\omega\in T^\ZZ$ and $n\geq 1$, the set
 $$
 H_\w^n=\{ x\in  M\colon \mbox{ $n$ is a
 $(\lambda,\delta)$-hyperbolic time for $(\omega,x)$ }\}.
 $$
 It  follows from the definition that if
$n$ is a $(\lambda,\delta)$-hyperbolic time for
$(\omega,x)\in T^\ZZ\times M$, then $(n-j)$ is a
$(\lambda,\delta)$-hyperbolic
time for $(\sigma^j(\w),f^j_\w(x))$, with $1\leq j< n$.

 \cle[Pliss] \label{l.pliss} Given $0<c\le A$   let $\zeta=c/A$.
Assume that $a_1, \ldots, a_N$ are real numbers satisfying $a_j
\le A $ for every $ 1\le j \le N $ and $
\sum_{j=1}^{N} a_j \ge cN.$ Then there are $\ell \ge \zeta N$
and $1 \le n_1 < \cdots < n_\ell \le N$ so that
$
\sum_{j=n}^{n_i} a_j\ge 0
$
for every $ 1 \le n \le n_i $ and  $1\le i\le  \ell$. \fle

\dem See \cite{ABV00}, Lemma 3.1.\cqd

\cde\label{de.freq} We say that the {\em frequency of
$(\lambda,\delta)$-hyperbolic times}\index{frequency} for
$(\omega,x)\in T^{\ZZ}\times M$ is larger than
$\zeta>0$ if, for large $n\in\NN$, there are $\ell \ge\zeta
n$ and integers $1\le n_1<n_2\dots <n_\ell\le n$ which are
$(\lambda,\delta)$-hyperbolic times for $(\omega,x)$.
 \fde

\cpr\label{pr.hyperbolic1} Assume that $f$ is
non-uniformly expanding on random
orbits. Then there are $0<\lambda<1$, $\delta>0$ and $\zeta>0$
(depending only on ${a_0}$ in \eqref{NUE} and on the map $f$)
such that for
all $\omega\in\Omega_\ep$ and Lebesgue almost every point $x\in M$, the
frequency of $(\lambda,\delta)$-hyperbolic times
for $(\w,x)$ is larger than $\zeta$. \fpr

\dem The proof follows from using
Lemma~\ref{l.pliss} twice, first for the sequence given by
$a_j=-\log\|Df_{\w_{j-1}}(f_\omega^{j-1}(x))^{-1}\|$  (up to a cut off that
makes it bounded from above in the presence of critical set), and then with
$a_j=\log\dist_\delta(f_\omega^{j-1}(x),\cc)$ for a convenient
choice of  $\delta>0$. We prove that there exist many times $n_i$
for which the conclusion of  Lemma~\ref{l.pliss} holds,
simultaneously, for both sequences. Then we check that any such
$n_i$ is a $(\lambda,\delta)$-hyperbolic time for $(\omega,x)$.

Assuming that \eqref{NUEOA2} and \eqref{RecLentOA} holds for $(\omega,x)$, then
for large
$N$ we have
\begin{equation*}
\sum_{j=1}^{N} -\log\|Df_{\w_{j-1}}(f_\omega^{j-1}(x))^{-1}\| \ge {a_0}
N\,.
\end{equation*}
If $\mathcal C=\emptyset$, we just use Lemma~\ref{l.pliss} for the
sequence
$$a_j=-\log\|Df_{\w_{j-1}}(f_\omega^{j-1}(x))^{-1}\|
-\frac{a_0}2,$$
with $c=a_0/2$ and $A=\max_{t\in\supp(\theta_\ep)}\max_{x\in
M}\{-\log\|Df_t(x)^{-1}\|-a_0/2\}$, we
obtain the result for $\zeta={a_0}/(2A)$ and $\lambda=e^{-{a_0}/2}$ ($\delta$
is not required in this case).

If $\mathcal C$ is not empty we recall assumption \eqref{e.perturbation}. Take
$B,\beta>0$ given by Definition~\ref{d.nd}.
Condition (c$_1$) implies that for
large $\rho>0$
\begin{equation}
\label{e.bound2} \left| \,\log \|Df(x)^{-1}\|\,\right| \le \rho \,
|\log \dist(x,\cc)|
\end{equation}
for every $x\in M\setminus \cc$. Fix $\alpha_1>0$ so that
$\rho\alpha_1 \le {a_0}/2$. The condition
\eqref{RecLentOA} of slow recurrence to $\cc$ ensures that we may
choose $r_1>0$   so that for large $N$
\begin{equation}
\label{e.bound1}
\sum_{j=1}^{N}\log\dist_{r_1}(f_\omega^{j-1}(x),\cc)\ge
-\alpha_1 N\,.
\end{equation}
 Fix any open neighborhood $V$ of $\cc$  and take $Q \ge \rho \, |\log r_1|$
large enough
so that it is also an upper bound for $-\log\|Df^{-1}\|$ on
$M\setminus V$. Then let
$$
J=\big\{1 \le j \le N\,\colon\,-\log\|Df(f_\omega^{j-1}(x))^{-1}\|
> Q\big\}.
$$
Note that if $j\in J$, then $f_\omega^{j-1}(x)\in V$. Moreover,
for each $j\in J$
$$
\rho \, |\log r_1| \le Q < -\log\|Df(f_\omega^{j-1}(x))^{-1}\| <
\rho \, |\log \dist(f_\omega^{j-1}(x),\cc)|,
$$
which shows that  $\dist(f_\omega^{j-1}(x),\cc) < r_1$ for every
$j\in J$. In particular,
$$
\dist_{r_1}(f_\omega^{j-1}(x),\cc)=\dist(f_\omega^{j-1}(x),\cc)<r_1,\quad\forall
j\in J.
$$
Therefore, by (\ref{e.bound2}) and (\ref{e.bound1}),
$$
\sum_{j\in J} -\log\|Df(f_\omega^{j-1}(x))^{-1}\| \le \rho
\sum_{j\in J} | \log \dist(f_\omega^{j-1}(x),\cc) | \le \rho \,
\alpha_1 N\le \frac{a_0}{2}N.
$$
%
Define
$$
b_j=\left\{\begin{array}{ll}
           -\log\|Df(f_\omega^{j-1}(x))^{-1}\|, & \text{if } j\notin J \\
           0 & \text{if } j\in J.
           \end{array}\right.
$$
By definition, $b_j\le Q$ for each $1\le j \le N$. As a
consequence,
$$
\sum_{j=1}^{N} b_j =  \sum_{j=1}^{N}
-\log\|Df(f_\omega^{j-1}(x))^{-1}\| -
                      \sum_{j\in J} -\log\|Df(f_\omega^{j-1}(x))^{-1}\|
                   \ge \frac{a_0}{2} N \,.
$$
Defining
 $a_j=b_j-{a_0}/4,$
we have
$$ \sum_{j=1}^{N}  a_j
                   \ge \frac{a_0}{4} N \,.
                   $$
Thus, we may apply Lemma~\ref{l.pliss} to $a_1,\dots,a_N$, with
$c={a_0} /4$ and $A=Q$. The lemma provides $\zeta_1>0$ and
$\ell_1\ge \zeta_1 N$ times $1 \le p_1< \cdots <p_{\ell_1}\le N$
such that for every $0 \le n \le p_i-1$ and $1\le i\le \ell_1$
\begin{equation}
\label{e.conclusion1} \sum_{j=n+1}^{p_i}
-\log\|Df(f_\omega^{j-1}(x))^{-1}\| \ge \sum_{j=n+1}^{p_i} b_j =
\sum_{j=n+1}^{p_i} \left(a_j +\frac{a_0}{4}\right)\ge
\frac{a_0}{4}(p_i-n)x.
\end{equation}

Now fix $\alpha_2>0$ small enough so that $\alpha_2 <
\zeta_1b {a_0}/4$, and let $r_2>0$ be such that
\begin{equation*}\label{e.e2r2}
\sum_{j=1}^{N} \log\dist_{r_2}(f_\omega^{j-1}(x),\cc)\ge -
\alpha_2 N \,.
\end{equation*}
Defining $a_j=\log\dist_{r_2}(f_\omega^{j-1}(x),\cc)+b{a_0}/4$
we have
      $$
\sum_{j=1}^{N}  a_j
                   \ge \left(\frac{bc_0}{4}-\alpha_2\right) N \,.
                   $$
Applying now Lemma~\ref{l.pliss} to $a_1,\dots,a_N$ with
$c=b{a_0}/{4}-\alpha_2$ and $A=b{a_0}/{4}$, we conclude
that there are $l_2\ge\zeta_2 N$ times $1\le q_1 < \cdots <
q_{\ell_2}\le N$ such that for every $0 \le n \le q_i-1$ and $1\le i \le \ell_2$
\begin{equation}
\label{e.conclusion2} \sum_{j=n+1}^{q_i}
\log\dist_{r_2}(f_\omega^{j-1}(x),\cc) \ge -\frac{b a_0}{4} \,
(q_i - n).
\end{equation}
Moreover,
$$
\zeta_2=\frac{c}{A}=1-\frac{4\alpha_2}{b a_0}.
$$
Finally, our condition on $\alpha_2$ means that
$\zeta_1+\zeta_2>1$. Let $\zeta=\zeta_1+\zeta_2-1$. Then
there exist $\ell=(\ell_1+\ell_2-N) \ge \zeta N$ times $1 \le
n_1< \cdots < n_\ell\le N$ at which (\ref{e.conclusion1}) and
(\ref{e.conclusion2}) occur simultaneously:
$$
\sum_{j=n}^{n_i-1} -\log\|Df(f_\omega^j(x))^{-1}\| \ge
\frac{a_0}4(n_i-n)
$$
and
$$
\sum_{j=n}^{n_i-1} \log\dist_{r_2}(f_\omega^j(x),\cc) \ge -\frac{b
{a_0}}4 (n_i - n),
$$
for every $0 \le n \le n_i-1$ and $1\le i \le \ell$. Letting
$\lambda=e^{-{a_0}/4}$ we easily obtain from the inequalities
above
$$
\prod_{j=n_i-k}^{n_i-1} \|Df(f_\omega^j(x))^{-1}\| \le \lambda^{k}
\quad\text{and}\quad \dist_{r_2}(f_\omega^{n_i-k}(x),\cc) \ge
\lambda^{b k},
$$
for every  $1\le i \le \ell$ and $1\le k\le n_i$. In other words,
all those $n_i$ are $(\lambda,\delta)$-hyperbolic times for
$(\omega,x)$, with $\delta=r_2$. \cqd

\cre\label{r.strongran2} In the presence of critical set, one
can sees that condition \eqref{RecLentOA} is not needed in all
its strength.
Actually, it is enough that \eqref{RecLentOA} holds for some
sufficiently small ${b_0}>0$  and some convenient
 $\delta>0$ (e.g. ${b_0}=\min\{\alpha_1,\alpha_2\}$ and $\delta=\max\{r_1,r_2\}$  in the proof of Proposition~\ref{pr.hyperbolic1}). \fre

\cre\label{r.start} Observe that the proof of
Proposition~\ref{pr.hyperbolic1} gives more precisely that if for
some $(\omega,x)\in T^\ZZ\times M$ and $N\in \NN$
\begin{equation*}
\sum_{j=0}^{N-1} -\log\|Df_{\sigma^j(\w)}(f_{\omega}^j(x))^{-1}\| \ge {a_0} N
\qand \sum_{j=0}^{N-1} \log\dist_{\delta}(f_{\omega}^j(x),\cc)\ge
- b_0 N \,
\end{equation*}
(where $b_0$ and $\delta$ are chosen according to
Remark~\ref{r.strongran2}), then there exist integers $0< n_1<
\cdots < n_l\le N$ with $l\ge \zeta N$ such that $n_i$ is a
$(\lambda,\delta)$-hyperbolic time for $(\omega,x)$, for each $1\le
i \le l$. \fre

The next result give us property $(m_1)$ at Section \ref{s.check}, and is needed
to
ensure later some metric estimates on the algorithm for the random induced partition.

\cle\label{c:hyperbolic3}
    Let $A\subset M$ be a set with positive
    Lebesgue measure, for whose points $x$ we have $(\omega,x)$ with
frequency of
    $(\lambda,\delta)$-hyperbolic times greater than $\zeta>0$, for all $\w\in\Omega_\ep$. Then there is
$n_0\in\NN$ such that for $n\ge n_0$
        $$
       \frac{1}{n}\sum_{j=1}^{n}\frac{\leb(A\cap H_\w^{j})}{\leb(A)} \geq
       \zeta.
       $$          \fle

 \dem Since we are assuming that points
$(\omega,x)$ for which $x$ is in $A$  have  frequency of
$(\lambda,\delta)$-hyperbolic times greater than
    $\zeta>0$, there is
    $n_0\in\NN$ such that for every $x\in A$ and
     $n\ge n_0$ there are $(\lambda,\delta)$-hyperbolic times
    $0<n_1<n_2<\dots <n_\ell\le n$  for $x$ with $\ell \ge \zeta n$.
     Take $n\ge n_0$ and let $\xi_n$ be the measure in $I_n=\{1,\dots,n\}$
defined
    by $\xi_n(J)=\# J/n$, for each  $J\subset I_n$. Then, putting $\chi(x,i)=1$ if $x\in H_\w^i$, and $\chi(x,i)=0$ otherwise, by Fubini's
    Theorem
    \begin{eqnarray*}
        \frac{1}{n} \sum_{j=1}^{n}\leb(A\cap H_\w^j)
        & =& \int \left(\int_A \chi(x,i)\,d\leb(x)\right)d\xi_n(i) \\
        & = & \int_A \left(\int \chi(x,i)\,d\xi_n(i)\right)d\leb(x).
    \end{eqnarray*}
    Since for every $x\in A$ and
     $n\ge n_0$ there are $0<n_1<n_2<\dots <n_\ell\le n$ with $\ell \ge\zeta n$
     such that $x\in H_\w^{n_i}$ for $1\le i\le\ell$, then the
    integral with respect to $d\xi_n$ is larger than $\zeta>0$ and
    the last expression in the formula above is bounded from below by
    $\zeta \leb(A)$. \cqd

\cle \label{l.contr}  Given $0<\lambda<1$ and $\delta>0$, there is
$\delta_1 > 0$ such that if $n$ is a $(\lambda,\delta)$-hyperbolic
time for~$(\omega,x)\in T^\ZZ\times M$, then
\begin{equation*}
\label{e.littleloss} \|Df_\w(y)^{-1}\| \le \lambda^{-1/2}
\|Df_\w(x)^{-1}\|,
\end{equation*}
for any  point $y$ in the ball of radius $\delta_1 \lambda^{n/2}$
around $x$. \fle
\dem

If $\mathcal C=\emptyset$
and since $f_\w$ is a local diffeomorphism, for each
$x\in M$ there is a
radius  $\delta_x>0$ such that $f_\w$ sends a neighborhood of $x$
diffeomorphically onto $B(f(x),\delta_x)$, the ball of radius $\delta_x$
around~$f(x)$. By compactness of $M$ we may choose a uniform radius
$\delta_1>0$. We choose $\delta_1>0$ small enough so that also
\begin{equation*}
\|Df_\w(y)^{-1}\| \le \lambda^{-1/2} \|Df_\w (x)^{-1}\|
,
\end{equation*}
whenever $y\in B(x,\delta_1\lambda^{1/2})$ and $\w\in\supp(\theta_\ep^\ZZ)$.

In the case $\mathcal C\neq\emptyset$, if $n$ is a
$(\lambda,\delta)$-hyperbolic time for
$(\omega,x)$, then
 $$\dist_\delta(x,\cc) \ge
\lambda^{bn}.$$ According to the definition of the truncated
distance, this means that
$$
\dist(x,\cc)= \dist_\delta(x,\cc) \ge \lambda^{bn}, \quad\text{or
else}\quad \dist(x,\cc) \ge \delta.
$$
In either case, considering $2\delta_1<\delta$, we have for any
point $y$ in the ball of radius $\delta_1 \lambda^{n/2}$ around $x$
 $$
\dist(y,x)<\frac12\dist(x,\cc) ,$$
because we haven chosen $b<1/2$ and $\delta_1 < \delta/2 < 1/2$
Therefore, we may use~(c$_2$) to conclude that
$$
\log\frac{\|Df(y)^{-1}\|}{\|Df(x)^{-1}\|} \le B \frac{\dist(y,x)}
           {\dist(x,\cc)^\beta}
\le B \frac{\delta_1 \lambda^{n/2}}
           {\min\{\lambda^{b\beta n},\delta^\beta\}}.
$$
Since $\delta>0,0<\lambda<1$ and we have taken $b\beta<1/2$, the term
on the right hand side is bounded by $B\delta_1 \delta^{-\beta}$.
Choosing  $\delta_1>0$ small so that $B\delta_1
\delta^{-\beta}<\log \lambda^{-1/2}$ we get the conclusion. \cqd

We assume that given $(\lambda, \delta)$ as before we fix
$\delta_1$ small so that Lemma~\ref{l.contr} holds. We further
require $\delta_1$ small so that the exponential map is an
isometry onto its image in the ball of radius $\delta_1$. This in
particular implies that any point in the boundary of a ball of
radius $\delta_1$ can be joined to the center of the ball through
a smooth curve of minimal length (a geodesic arc).

 \cpr\label{p.contr}
    If $n$ is
    $(\lambda,\delta)$-hyperbolic time for $(\omega,x)\in T^\ZZ\times M$,
    then there is a neighborhood $V^n_\omega(x)$ of $x$ in $M$ such that:
     \begin{enumerate}
     \item $f_{\omega}^n$ maps $V^n_\omega(x)$
    diffeomorphically onto $B(f_{\omega}^n(x),{\delta_1})$;
     \item\label{cond2} for every $y \in V^n_\omega(x)$ and $1\leq k\leq n$ we
     have $  \| Df_{\sigma^{n-k}(\w)}^k(f_\omega^{n-k}(y))^{-1} \| \le
\lambda^{k/2};$
\item for every $y,z\in V^n_\omega(x)$ and $1\leq k\leq n$
     $$ \dist(f_{\omega}^{n-k}(y),f_{\omega}^{n-k}(z)) \leq \lambda^{k/2}\dist(f_{\omega}^{n}(y),f_{\omega}^{n}(z)).
     $$
      \end{enumerate}

\fpr

\dem We shall prove the first two items by induction on $n$. Let
$n=1$ be a $(\lambda,\delta)$-hyperbolic time for
$(\omega,x)\in T^\ZZ\times M$. It follows from
Lemma~\ref{l.contr} and from the definition of hyperbolic times that for any $y$
in the ball $B(x,{\delta_1
\lambda^{1/2}})\subset M$ of radius $\delta_1 \lambda^{1/2}$
around~$x$
\begin{equation}\label{eq.dilat.preb}
   \|Df_\w(y)^{-1}\| \le \lambda^{-1/2} \|Df_\w(x)^{-1}\| \le \lambda^{1/2}.
\end{equation}
 This means that $f_\w$ is a
$\lambda^{-1/2}$-dilation in the ball $B(x,{\delta_1
\lambda^{1/2}})$. Then,  there exists some neighborhood
$V^1_\omega(x)$ of $x$ contained in $B(x,{\delta_1 \lambda^{1/2}})$
which is mapped diffeomorphically onto the ball
$B(f_\omega(x),{\delta_1})$ and for $y\in V^1_\omega(x)$ condition
\eqref{eq.dilat.preb} ensures the second property $\|Df_\w(y)^{-1}\| \le
\lambda^{1/2}$.

Assume now that the conclusion holds for $n\ge 1$. Let $n+1$ be a
$(\lambda,\delta)$-hyperbolic time for $(\omega,x)\in
T^\ZZ\times M$.
Take any $z\in \partial B(f_\omega^{n+1}(x),{\delta_1 })$, and let  $\gamma\colon [0,1]\to M$ be a smooth curve of minimal length joining $z$ to $f_\omega^{n+1} (x)$.
The curve $\gamma$ necessarily lies  inside $B(f_\omega^{n+1}(x),{\delta_1 })$ by the choice of $\delta_1$. 
Consider $\gamma_n$ and $\gamma_{n+1}$ smooth curves which are
lifts of $\gamma$ starting at $f_\omega(x)$ and $x$, respectively.
This means that
 $$ \gamma=f_{\sigma(\omega)}^{n}\circ \gamma_{n}\qand \gamma=
f_\omega^{n+1}\circ \gamma_{n+1},$$
at least in the domains where the lifts make sense. Since $n$ is a
$(\lambda,\delta)$-hyperbolic time for
$(\sigma(\omega),f_\omega(x))$, by induction hypothesis there is a
neighborhood $V_{\sigma(\omega)}^{n}(f_\omega(x))$ which is sent
diffeomorphically by $f_{\sigma(\omega)}^{n}$ onto the ball of
radius $\delta_1$ around $f_\omega^{n+1}(x)$ with the additional
second condition property. One has
that $\gamma_{n}$ lies inside $
V^{n}_{\sigma(\omega)}(f_\omega(x))$.

Moreover, $n-j$ is a $(\lambda,\delta)$-hyperbolic time
for
$(\sigma^{j+1}(\w), f_\w^{j+1}(x))$ and
$f_{\sigma(\w)}^j(V^{n}_{\sigma(\omega)}(f_\omega(x)))$ is
a neighborhood $V_{\sigma^{j+1}(\w)}^{n-j}$ of $f_{\w}^{j+1}(x)$ which is
mapped by $f_{\sigma^{j+1}(\w)}^{n-j}$ diffeomorphically onto
$B(f_\w^{n+1}(x),{\delta_1})$ and also satisfies the second
condition property, for $1\leq j \leq n$.

\smallskip

\noindent {\bf Claim.} {\em The curve $\gamma_{n+1}$ lies inside
the ball of radius $\delta_1 \lambda^{(n+1)/2}$ around $x$.}

\smallskip \noindent
Assume, by contradiction, that $\gamma_{n+1}$ hits the boundary of
$B(x,{\delta_1 \lambda^{(n+1)/2}})$ before the end time. Let
$0<t_0<1$ be the first moment in such conditions. One necessarily
has that $\gamma_{n+1}\vert [0,t_0]$ is a curve inside the ball
$B(x,{\delta_1 \lambda^{(n+1)/2}})$ joining $x$ to a point in the
boundary of that ball. Since $n+1-j$ is a $(\lambda,\delta)$-hyperbolic
time for $(\sigma^{j+1}(\w), f_\w^{j+1}(x))$, by Lemma \ref{l.contr} that for
each $0\le
t\le t_0$ and $0\leq j\leq n$
$$\|Df_{\sigma^j(\w)}(f_\w^j(\gamma_{n+1}(t)))^{-1}\|\leq\lambda^{-1/2}\|Df_{\sigma^j(\w)}
(f_\w^j(x))^{-1} \|$$
and
$f_\w^j(\gamma_{n+1}([0,t_0]))\subset
V_{\sigma^j(\omega)}^{n+1-j}(f_\omega^j(x)),$ which yields to
\begin{eqnarray}
  \|Df_\w^{n+1}(\gamma_{n+1}(t))^{-1}\| &\le &
\prod_{j=0}^n\|Df_{\sigma^{n-j}(\w)}(\gamma_{j+1}(t))^{-1}\| \nonumber\\
    &=& \|Df_{\w}(\gamma_{n+1}(t))^{-1}\|
\cdot\prod_{j=0}^{n-1}\|Df_{\sigma^{n-j}(\w)}(\gamma_{j+1}(t))^{-1}\|
\nonumber\\
     &\le &
\lambda^{-1/2}\|Df_{\w}(x)^{-1}\|\cdot\prod_{j=0}^{n-1}\lambda^{-1/2}
\|Df_{\sigma^{n-j}(\w)}(f_\w^{n-j}
(x))^{-1}\| \nonumber\\
     &\le & \lambda^{-1/2}\lambda\lambda^{-n/2}\lambda^n\nonumber\\
     &= & \lambda^{(n+1)/2}.\label{eq.expa}
\end{eqnarray}
Hence
 \begin{eqnarray*}
   \int_0^{t_0}\|\gamma'(t)\|dt &=&
\int_0^{t_0}\|Df_\w^{n+1}(\gamma_{n+1}(t))\cdot\gamma_{n+1}'(t)\|dt \\
   &\ge &  \int_0^{t_0}\lambda^{-(n+1)/2}\|\gamma'_{n+1}(t)\|dt\\
   &=& \delta_1
 \end{eqnarray*}
This gives a contradiction since $t_0<1$, thus proving the claim.

\smallskip

Let us now  finish  the proof of the first two items. We simply consider the
lifts by $f_\omega^{n+1}$ of the geodesics joining
$f_\omega^{n+1}(x)$ to the points in the boundary  of $B(f_\omega^{n+1}(x),{\delta_1
})$. This defines a neighborhood $V^{n+1}_\w(x)$
of $x$ which by \eqref{eq.expa}~has the required properties.

For the third item, let $\gamma$ be a curve of minimal length connecting
$f_\w^n(z)$ to $f_\w^n(y)$. This curve $\gamma$ must obviously
be contained in $B(f_\w^n(x),{\delta_1})$. For $1\le k \le n$,
let~$\gamma_{k}$ be the (unique) curve in
$f_\w^{n-k}(V_\w^n(x))$ joining $f_\w^{n-k}(z)$ to $f_\w^{n-k}(y)$
such that $f_{\sigma^{n-k}(\w)}^{k} ( \gamma_{k})=\gamma$. We
have for every $n\ge 1$
\begin{eqnarray*}
  \length(\gamma) &=& \int\|\gamma'(t)\|dt \\
    &=&  \int\|Df_{\sigma^{n-k}(\w)}^{k}( \gamma_{k}(t))\cdot
\gamma_{k}'(t)\|dt\\
    &\ge& \lambda^{-\frac{k}2}\int\| \gamma_{k}'(t)\|dt\\
    &=&\lambda^{-\frac{k}2} \length( \gamma_{k}).
\end{eqnarray*}
As a consequence,
$$
\dist(f_\w^{n-k}(y),f_\w^{n-k}(z))\le\length( \gamma_{k}) \le
\lambda^{\frac{k}2}\length(\gamma)
=\lambda^{\frac{k}2}\dist(f_\w^n(y),f_\w^n(z)).
$$\cqd

\cco[Bounded Distortion]
\label{co.distortion0}
  There exists $C_0>0$ such that if
 $n$ is a $(\lambda,\delta)$-hyperbolic time for $(\omega,x)\in
T^\ZZ\times M$, then for every $y, z\in V^n_\omega(x)$,
 $$ \log\frac{|\det Df_{\omega}^n (y)|}{|\det Df_{\omega}^n (z)|}
 \le C_0\dist(f_{\omega}^{n}(y),f_{\omega}^{n}(z)).
 $$
  \fco
\dem
 If $\mathcal C=\emptyset$, since $f_\w\in C^2(M,M)$ then there is $C_0'>0$ such that
for all $z,y\in M$ and $\w\in\supp(\theta_\ep)$ we have $
 |\log
|\det Df_\w(z)|-\log |\det Df_\w(y)||\le C_0'\dist (z,y). $
And  for all $z,y\in V_n(x)$ we have
\begin{eqnarray*}
  \log\frac{|\det Df_w^n(z)|}{|\det Df_\w^n (y)|} &=& \sum_{j=0}^{n-1}\log\frac{|\det Df_{\sigma^j(\w)} (f_\w^j(z))|}{|\det Df_{\sigma^j(\w)}(f_\w^j (y))|} \\
  &\le& \sum_{j=0}^{n-1}C_0'\dist(f_\w^j(z),f_\w^j(y))\\
    &\le& \sum_{j=0}^{n-1}C_0'\lambda^{n-j}\dist(f_\w^n(z),f_\w^n(y)).
\end{eqnarray*}
It is then enough to take
$C_0=\exp\left(\sum_{k=0}^{\infty}C_0'\lambda^{k}\right)$.

If $\mathcal C$ is not empty then let $n$ be a
    $(\lambda,\delta)$-hyperbolic time for $(\omega,x)\in T^{\ZZ}\times M$
with associated hyperbolic
    pre-ball $V^n_\omega$. By Proposition~\ref{p.contr} we have   for
    each $y,z\in V_\w^n$ and each  $0\le k \leq n-1$
     $$\dist(f_{\omega}^k(y),f_{\omega}^k(z))\le \delta_1 \lambda^{(n-k)/2}.$$
     On the other hand, since $n$ is a hyperbolic time for $(\omega,x)$
  \begin{eqnarray}
  \dist (f_{\omega}^k(y),\cc)
  &\ge &\dist(f_{\omega}^k(x),\cc)-\dist(f_{\omega}^k(x),f^k(y))\nonumber\\
  &\ge &\lambda^{b(n-k)}-\delta_1 \lambda^{(n-k)/2}\nonumber\\
  &\ge & \frac12 \lambda^{b(n-k)}\label{in}\\
  &\ge &2\delta_1 \lambda^{(n-k)/2},\nonumber
  \end{eqnarray}
  as long as $\delta_1<1/4$; recall that $b< 1/2$. Thus we have
  $$\dist(f_{\omega}^k(y),f_{\omega}^k(z))\le \frac12\dist (f^k_{\omega}(y),\cc),$$
and so we may use (c$_3$) to obtain
 $$
\log\frac{|\det Df (f_{\omega}^k(y))|}{|\det Df(f_{\omega}^k
(z))|}\le
\frac{B}{\dist(f_{\omega}^k(y),\cc)^\beta}\dist(f_{\omega}^k(y),f_{\omega}^k(z)).
$$
Hence, by \eqref{in} and Preposition~\ref{p.contr}
\begin{eqnarray*}
  \log\frac{|\det Df_{\omega}^n (y)|}{|\det Df_{\omega}^n (z)|}
  &=& \sum_{k=0}^{n-1}\log\frac{|\det Df(f_{\omega}^k(y))|}{|\det Df(f_{\omega}^k(z))|} \\
  &\le & \sum_{k=0}^{n-1}2^\beta  B   \frac{\lambda^{(n-k)/2}}
                             {\lambda^{b\beta(n-k)}}\dist(f_\w^n(y),f^n_\w(z)).
\end{eqnarray*}
It suffices to take $C_0 \ge \sum_{k=1}^{{+\infty}} 2^\beta B
\lambda^{(1/2-b\beta)k}$; recall that $b\beta < 1/2$. \cqd

We shall often refer to the sets $ V^n_\omega$ as \emph{hyperbolic
pre-balls} and to their images $ f_\w^{n}(V^n_\omega) $ as
\emph{hyperbolic balls}.  Notice that the latter are indeed balls
of radius $ \delta_1>0 $.

%
%

Many times along this text it will be useful to have the following
weaker forms of the previous corollary.

\cco\label{co.distortion1}
 There is a constant $C_1>0$ such that if
 $n$ is a $(\lambda,\delta)$-hyperbolic time for $(\omega,x)\in
T^\ZZ\times M$ and $y,z\in V^n_\omega(x)$, then
 $$
 \frac{1}{C_1}\leq \frac{|\det Df_{\omega}^n(y)|}{|\det Df_{\omega}^n(z)|} \leq C_1.
 $$
 \fco
  \dem
By Corollary \ref{co.distortion0} just have to consider
$C_1=\exp(C_0\delta_1)$.
 \cqd

\cco\label{co.distortion2}
There is a constant $C_2>0$ such that for any hyperbolic pre-ball $V_\w^n(x)$
and any $A_1, A_2\subset V_\w^n(x)$
$$\frac{1}{C_2}\frac{\leb(A_1)}{\leb(A_2)}\leq\frac{\leb(f_\w^n(A_1))}{
\leb(f_\w^n(A_2))}\leq{C_2}\frac{\leb(A_1)}{\leb(A_2)}.$$

\begin{proof}
By the change of variable formula for $f_\w^n$ we may write
\begin{eqnarray*}\frac{\leb(f_\w^n(A_1))}{
\leb(f_\w^n(A_2))}&=&\frac{\int_{A_1}|\det Df_\w^n(z)|d\leb(z)}{\int_{A_1}
|\det Df_\w^n(z)|d\leb(z)}\\
&=&\frac{|\det Df_\w^n(z_1)|\int_{A_1}\left|\frac{\det
Df_\w^n(z)}{\det Df_\w^n(z_1)}\right|d\leb(z)}{|\det
Df_\w^n(z_2)|\int_{A_2}\left|\frac{\det
Df_\w^n(z)}{\det Df_\w^n(z_2)}\right|d\leb(z)},
\end{eqnarray*}
with $z_1$ and $z_2$ choosen arbitrarily in $A_1$ and $A_2$, respectively. From
Corollary
\ref{co.distortion1} we get the desired bounds.
\end{proof}

\fco

\subsection{Transitivity and growing to large scale} We do not need
transitivity of $f$ in all its strength.
Before we tell
 what is the weaker form of transitivity  that is enough for our
purposes, let us refer that given $\delta>0$, a subset $A$ of $M$
is said to be {\em $\delta$-dense}\index{$\delta$-dense} if any
point in $M$ is at a distance smaller than $\delta$ from $A$. For
our purposes it is enough that there is some
point $p\in M$ whose pre-orbit does not hit the critical set of
$f$ and is $\delta$-dense for some sufficiently small $\delta>0$
(depending on the radius $\delta_1$ of hyperbolic balls for $f$). As the
lemma below shows, in our setting of non-uniformly expanding maps
this is a consequence of the usual transitivity of $f$.

\cle Let $f\colon M\to M$ be a transitive
non-uniformly expanding map. Given $\delta>0$ there is $p\in M$
and $N_0\in\NN$ such that $\cup_{j=0}^{N_0}f^{-j}\{p\}$ is
$\delta$-dense in $M$ and disjoint from the critical set $\cc$.
\fle

\dem See \cite{ALP05}, Lemma 2.5.
 \cqd

 Assuming that $f$ is non-uniformly expanding and non-uniformly expanding on
random orbits, then by
Proposition~\ref{pr.hyperbolic1} there are $\lambda$, $\delta$ and
$\zeta$ such that Lebesgue almost every $x\in M$ has frequency of
$(\lambda,\delta)$-hyperbolic times greater than $\zeta$.
We fix once
and for all $ p \in M $ and $N_0\in\NN$ for which
\begin{equation*}
\cup_{j=0}^{N_0}f^{-j}\{p\}\quad\text{is $\delta_1/3$-dense in $ M
$ and disjoint from $ \cc $,}
 \end{equation*}
  where
$\delta_1>0$ is the radius of hyperbolic balls as for
Proposition~\ref{p.contr}. Take constants $\alpha>0$ and
$\delta_0>0$ so that
 $$
 2\sqrt\delta_0 \ll \delta_1\qand 0< \alpha \ll \delta_0
 .
 $$
\begin{Lemma}\label{le:grow2}
There are constants $ K_0,  D_{0} 
>0$ depending only on $ f$, $\lambda$, $\delta_1 $ and the point $
p$ such that, if $\ep>0$ is sufficiently small, then for any ball $
B\subset M $ of radius $ \delta_1 $ and every $\omega\in
\supp(\theta_\ep^\ZZ)$
there are an open set $ A\subset B $ and an integer $
0\leq m \leq
 N_{0} $ for which:
   \begin{enumerate}
   \item $ f_{\omega}^{m} $ maps $A$ diffeomorphically onto
$B(p,2\sqrt\delta_{0}) $;
   \item
   for each $x,y\in A$
   $$
\log\left|\frac{\det Df_\w^{m}(x)}{\det Df_\w^{m}(y)}\right|
 \le
D_0 \dist(f_{\omega}^{m}(x),f_{\omega}^{m}(y));
$$
\end{enumerate}
and, moreover, for each  $ 0\leq j \leq
 N_{0} $  the $j$-preimages
$(f_{\omega}^{j})^{-1}B(p,2\sqrt\delta_{0})$
       are all disjoint from $\cc$, and
  for $x$
       belonging to any such $j$-preimage we have
  $$
 \frac1{K_0}\le \|Df_{\omega}^j(x)\|
 \le K_0.
 $$
\end{Lemma}

\begin{proof}
Since $\cup_{j=0}^{N_0}f^{-j}\{p\}$ is $ \delta_1/3 $ dense in $ M
$ and disjoint from $\cc$, choosing $ \delta_{0}>0 $ sufficiently
small we have that each connected component of the $j$-preimages
 $f^{-j}B(p, 2\sqrt\delta_{0})$,
 with $j\leq N_{0} $, are bounded away
from the critical set $ \cc $ and  contained in a ball of
radius $ \delta_1/3 $. {Moreover, since we are dealing with a finite
number of iterates, less than $N_0$,
and $f_t$ varies
continuously with parameter $t$, if $\ep$ is sufficiently
small then for every $\omega\in
\supp(\theta_\ep^\ZZ)$,
each connected component of the $j$-preimages
$(f^j_{\omega})^{-1}B(p, 2\sqrt\delta_{0})$,
 with $j\leq N_{0}$, is  uniformly (on $j$ and $\w$) bounded away
from the critical set $ \cc $  and  contained in a ball of
radius $ \delta_1/3 $.}
This immediately implies that for $\w\in\supp(\theta_\ep^\ZZ)$, any ball
$ B \subset M $ of radius $ \delta_1 $ contains a $m$-preimage $
A $ of $ B(p, 2\sqrt\delta_{0}) $ which is mapped
diffeomorphically by $f^m_{\omega}$ onto $ B(p, 2\sqrt\delta_{0})
$ for some $m\leq N_{0}$. Since the number of
iterations and the distance to the critical region are uniformly
bounded, the volume distortion is uniformly bounded and moreover there is some constant
$K_0>1$ 
such that for every $\w\in\supp(\theta_\ep^\ZZ)$
 $$
 \frac1{K_0}\le \|Df_{\omega}^m(x)\|
 \le K_0.
 $$
 for all $1\le m\le N_0$ and $x$ belonging to a $m$-preimage of $B(p,
2\sqrt\delta_{0})$ by $f^m_\omega$.
\end{proof}

Next we prove a useful consequence of the
existence of hyperbolic times, namely that if we start with some fixed given $
\alpha
> 0 $ then there exist some $ N_{\alpha} $ depending only
on $ \alpha $ such that, for $\w\in\Omega_\ep$, any ball on $M$ of
radius $ \alpha $ has some subset
which grows to a fixed size with bounded distortion within $
N_{\alpha} $ iterates.

\cle\label{le:grow1}
     Given $\alpha>0$ there exists $  N_{\alpha}
    > 0 $    such that if $\ep$ is
sufficiently small, then for every $\omega\in \Omega_\ep$ we have that any
ball $ B
    \subset M $ of radius $\alpha$ contains a hyperbolic pre-ball $
    V^{n}_\omega\subset B $ with $ n\leq  N_{\alpha} $.
\fle

\begin{proof}
    Take any $ \alpha > 0 $ and a ball
$B(z,\alpha)$. By
     Proposition~\ref{p.contr} we may choose
$n_{\alpha}\in\NN$
    large enough so that any hyperbolic pre-ball $ V^n_\omega $ associated to
    a
    hyperbolic time $ n\geq  n_{\alpha} $ will have diameter not exceeding
     $\alpha/2$.  Now
    notice that by Proposition~\ref{pr.hyperbolic1} for Lebesgue almost every
$x\in M$, the point $~(\omega^*,x)$ has an infinite
number of hyperbolic times and
    therefore
    $$
    \leb\left(M\setminus \mcup_{j= n_{\alpha}}^{n}H_{\w^*}^{j}\right) \to 0
    \quad\text{ as } n\to{+\infty}.
    $$
    Hence, it is possible to choose
$N_{\alpha}\in \NN$ such
    that
   \begin{equation}\label{Neps}
    \leb\left(M\setminus \mcup_{j=
    n_{\alpha}}^{N_\alpha}H_ {\w^*}^{j}\right)
   < m(B(z, \alpha/2)).
    \end{equation}
Observe that if $n$ is a hyperbolic
time for $(\omega^*,x)$ and $\ep$ is small enough, then for every
$\omega\in \Omega_\ep$ the natural $n$ is also a hyperbolic time for
$(\omega,x)$. Hence, if $\ep$ is small enough, for given
$\alpha>0$ we can take an integer $N_\alpha$, only depending on
$\alpha$,
$\lambda$ and $\delta_1$, such that  \eqref{Neps} holds for every
$\w\in\Omega_\ep$ in the place of $\w^*$.
 This ensures that, for every $\w\in\Omega_\ep$, there is
    a point $ \hat x \in B(z, \alpha/2) $ with a hyperbolic time
    $ n\leq  N_{\alpha} $ and associated hyperbolic pre-ball $
    V^{n}_\omega(x)$ contained in $ B(z, \alpha) $.\end{proof}

\subsection{The partitioning algorithm}\label{s.algoritmo}
We describe now the construction of the  partition
$\mathcal{P}_\omega$ (mod 0) of $ \Delta_{0}=B(p,\delta_{0}) $,
for every $\omega\in \Omega_\ep$. The basic intuition is that we
wait for some iterate $ f_{\omega}^{k}(\Delta_{0}) $ to cover $
\Delta_{0} $ completely, and then define the subset $ U
\subset \Delta_{0} $, for which $ f_{\omega}^{k}: U\to
\Delta_{0} $ is a diffeomorphism, as an element with return time $ k $ for the partitions corresponding
 to all
elements
$\w'\in\Omega_\ep$  with same first $k$ coordinates as $\w$: $\w_0'=\w_0,
\ldots,\w'_{k-1}=\w_{k-1}$.
After that, we
continue to iterate the complement $
\Delta_{0}\setminus U $ until this complement covers again
$ \Delta_{0} $ and repeat the same procedure to define more
elements of the final partition with higher return times.  Using
the fact that small regions eventually become large due to the
expansivity condition, it follows that this process can be
continued and that Lebesgue almost every point eventually belongs
to some element of the partition. Moreover, the return time
function depends on the time that it takes small regions to
become large on average and this
turns out to depend precisely on the measure of the tail set. On the other hand, this process avoids undesirable
randomness on the
{\em choice} of elements
for distinct (but related) partitions of the induced regions. In particular,
 for different realizations
with similar initial nonnegative
coordinates, the elements in corresponding partitions with return times lower
than the number of similar entries are the same, as subsets of $\Delta_0$.
\\

Now we introduce
 neighborhoods of $ p $
$$\Delta_{0}=
B(p,\delta_{0}),\quad \Delta^{1}_{0}=B(p,2\delta_{0}),\quad
\Delta^{2}_{0}=B(p,\sqrt\delta_{0})\qand
\Delta^{3}_{0}=B(p,2\sqrt\delta_{0}).$$ For $0<\lambda<1$ given by
Proposition~\ref{pr.hyperbolic1}, let
$$
I_{s}=\left\{x\in\Delta^{1}_{0}\: : \:\delta_{0}(1+\lambda^{s/2}) <
\dist(x,p) < \delta_{0}(1+\lambda^{(s-1)/2})\right\},\quad  s\ge 1,
$$
be a partition (mod 0) into countably many rings of $
\Delta_0^{1}\setminus \Delta_0 $.

The construction of the partition $\mathcal{P}_{\omega}$ of
$\Delta_0$ is inductive and we give the initial and the general
step of the induction.  For the sake of a better visualization of
the process, and to motivate the definitions, we start with the
first step.
 Define $$[\w]_k=\{\tau\in \Omega_\ep:\omega_0=\tau_0, \ldots,
\omega_{k-1}=\tau_{k-1}\}.$$
The set $H_\tau^k$ is the same for any $\tau\in[\w]_k$ and we will
refer to this set as $H_{[\w]}^k$.
\smallskip
\subsubsection*{First step of the induction}
 Take $ R_{0} $ some large integer to be determined in Section
 \ref{sub.ebd} (can
 be taken independent of $\omega$); we ignore any dynamics
occurring up to time $ R_{0} $. For $\w\in\Omega_\ep$, let $ k\geq R_{0}+1 $
be the first
time that $ \Delta_0\cap H^{k}_{[\w]}\neq\emptyset $. For $ j
< k $ we define formally the objects $
\Lambda_\w^j,
A_\w^{j},
A_\w^{j, \alpha}$, whose meaning will become clear in the next paragraphs,
by
$
A_\w^{j}=A_\w^{j,\alpha}=\Lambda_\w^j=\Delta_{0} $.

Let
$(U_{k,j}^3)_j$ be the
connected components of
 $
 A_\w^{k-1,\alpha}\cap(f_{w}^{k})^{-1}(\Delta_0^3)
 $
 contained in hyperbolic pre-balls $V^{k-m}_{\w}$, with $k- N_0\le
 m\le k$. This hyperbolic pre-balls $V^{k-m}_{\w}$ growth in
$k-m$ iterates to a
hyperbolic ball of radius $\delta_1$ which
$f_{\sigma^{k-m}(\w)}^m$ maps diffeomorphically onto
$\Delta_0^3$.
Now let
$$
U_{k,j}^i=U_{k,j}^3\cap
(f_{\w}^{k})^{-1}(\Delta_0^i),\quad i=0,1,2,\quad\text{where }
\Delta_0^0=\Delta_0,
$$
and set $R_{\omega}(x)=k$ for $x\in U_{k,j}^0$. Take
$$
\Lambda_w^{k}=\Lambda_w^{k-1}\setminus \{R_{\w}=k\}.
$$

We define also a function $ t_\w^{k}:\Lambda_w^{k}\to
\mathbb N $ by
\begin{equation*}
    t_\w^{k}(x) =
    \begin{cases}
    s & \text{ if }  x\in U_{k,j}^{1} \text{ and }
    f_{\w}^{k}(x) \in I_{s} \text{  for some $j$;} \\
    0 & \text{ otherwise}.
    \end{cases}
\end{equation*}
Let
$$
A_\w^{k}= \{x\in\Lambda_w^{k}: t_\w^{k}(x) = 0\}, \quad B_\w^{k}=
\{x\in\Lambda_w^{k}: t_\w^{k}(x) > 0\}.
$$
We also define:
 $$
 A_\w^{k,\alpha}=\Lambda_w^{k}\cap \bigcup_{x\in A_\w^k\cap H_\w^{k+1}}
 (f_{\omega}^{k+1}\vert_{V^{k+1}_\w(x)})^{-1}B(f_\w^{k+1}(x),\alpha)$$

For all $\tau\in[\w]_k$ and $j\leq k$, we define the objects $A_\tau^j,
B_\tau^j, \Lambda_\tau^k, \{R_\tau=j\}$, $t_\tau^k$ to be the same as the
corresponding ones  as before. Moreover, for $\tau\in[\w]_{k+1}$ we also define
$A_{\tau}^{k, \alpha}$ to be the same as $A_{\w}^{k, \alpha}$.

\smallskip
\subsubsection*{General step of the induction}
The general inductive step of the construction now follows by
repeating the arguments above with minor modifications.  More
precisely we assume that the
sets $\Lambda_\tau^{i}$, $A_\tau^{i}$,
 $B_\tau^{i}$, $\{R_{\tau}=i\}$ and functions $ t_\tau^{i}:
\Lambda_\tau^{i}\to\mathbb N $ are defined for all $ i\leq n-1 $ and are
exactly the same for every $\tau\in[\w]_{n-1}$. We also have defined
$A_\tau^{i,\alpha}$, for $i\leq n-1$ to be the same set for all
$\tau\in[\w]_{n}$. For $ i\leq
R_{0} $ we just let $ A_\w^{i}=A_\w^{i,\alpha}=\Lambda_w^{i}= \Delta_{0}$, $
B_\w^{i}=\{R_{\omega}=i\}=\emptyset $ and $ t_\w^{i}\equiv 0 $. Now, let
$(U_{n,j}^3)_j$ be the connected
components of
 $
 A_\w^{n-1,\alpha}\cap(f_{\omega}^n)^{-1}(\Delta_0)
 $
 contained in hyperbolic pre-balls $V^{r}_\omega$, with $n-  N_0\le
 r\le n$, which are mapped onto $\Delta_0^3$ by $f_{\omega}^n$.
 Take
$$
U_{n,j}^i=U_{n,j}^3\cap (f_{\omega}^{n})^{-1}(\Delta_0^i),\quad i=0,1,2,
$$
and set $R_{\omega}(x)=n$ for $x\in U_{n,j}^0$. Take also
$$
\Lambda_w^{n}=\Lambda_w^{n-1}\setminus \{R_{\omega}=n\}.
$$
The definition of the function $ t_\w^{n}:\Lambda_w^{n}\to \mathbb N $
is slightly different in the general case:
\begin{equation*}
    t_\w^{n}(x) =
    \begin{cases}
    s & \text{ if } x\in U_{n,j}^{1}\setminus U_{n,j}^{0} \text{ and }
    f_{\omega}^{n}(x) \in I_{s} \text{ for some $j$,} \\
    0 & \text{ if } x\in A_\w^{n-1} \setminus \cup_{j} U^{1}_{n,j},\\
    t_\w^{n-1}(x)-1 & \text{ if } x\in B_\w^{n-1}\setminus \cup_{j} U^{1}_{n,j}.
    \end{cases}
\end{equation*}
Finally let
$$
A_\w^{n}= \{x\in\Lambda_w^{n}: t_\w^{n}(x) = 0\}, \quad B_\w^{n}=
\{x\in\Lambda_w^{n}: t_\w^{n}(x) > 0\}.
$$
and
 $$
 A_\w^{n,\alpha}=\Lambda_w^{n}\cap \bigcup_{x\in A_\w^n\cap H^{n+1}_\w}
 (f_{\omega}^{n+1}\vert_{V^{n+1}_\w(x)})^{-1}B(f_\w^{n+1}(x),\alpha)$$
Once more, for all $\tau\in[\w]_n$  we define the objects $A_\tau^n,
B_\tau^n, \Lambda_\tau^n, \{R_\tau=n\}$, $t_\tau^n$ to be,
respectively, $A_\w^n,
B_\w^n, \Lambda_\w^n, \{R_\w=n\}$, $t_\w^n$ and for $\tau\in[\w]_{n+1}$
we also define
$A_{\tau}^{n, \alpha}$ as $A_{\w}^{n, \alpha}$.

\medskip

\cre\label{r.claim}
    Associated to each component $U^0_{n-k}$ of
    $\{R_\omega=n-k\}$, for some $k>0$, we have a collar $U^1_{n-k}\setminus
    U^0_{n-k}$ around it; knowing that the new components of $\{R_\omega=n\}$
    do not intersect ``too much" $U^1_{n-k}\setminus U^0_{n-k}$ is
    important for preventing overlaps on sets of the partition. By Lemma 4.5 in \cite{A04} it is enough to consider
     \begin{equation*}
     \alpha<K_0^{-1}\lambda^{N_0/2}\delta_{0} (\lambda^{-1/2}-1)
     \end{equation*} so that
 $U_{n}^1\cap\{t_\w^{n-1}>1\}=\emptyset$ for each component $U^1_n$.
 \fre

  \subsection{Expansion, bounded distortion and uniformity}\label{sub.ebd}
The inductive construction we detailed before provides a family of topological balls
 contained in $ \Delta_0 $ which, as we will see, define a Lebesgue modulo zero partition
$\mathcal{P}_\w$ of $\Delta_0$. We start however, by checking conditions $(1)$-$(3)$ in
the definition of the induced piecewise expanding Gibbs-Markov map in view to prove Theorem~\ref{random Markov
structure}.

Recall that by construction, the return time $ R_\omega $ for an
element $ U $ of the partition $ \mathcal P_\omega $ of $ \Delta_0
$ is formed by a certain number $ n $ of iterations given by the
hyperbolic time of a hyperbolic pre-ball $ V^n_\omega\supset U $,
and a certain number $ m\leq N_{0} $ of additional iterates which
is the time it takes to go from $ f_\omega^{n}(V^{n}_\omega)$,
which could be anywhere in $ M $, to $
f_\omega^{n+m}(V^{n}_\omega) $ which covers $ \Delta_0 $
completely. The map $F_\w=f_\w^{R_w}:\Delta\to\Delta$ is indeed a $C^2$
diffeomorphism from each component $U$ onto $\Delta$.

It follows from
Proposition~\ref{p.contr} and
Lemma~\ref{le:grow2} that
$$
\|Df_\omega^{n+m}(x)^{-1}\|\le
\|Df_{\sigma^n(\omega)}^{m}(f^n_\w(x))^{-1}\|\cdot\|Df_\omega^{n}(x)^{-1}
\|<K_0\lambda^{n/2}\le
K_0\lambda^{(R_0-N_0)/2}.
$$
By taking $ R_{0} $ sufficiently large we can make this last
expression smaller than some $\kappa_\w$, with $0<\kappa_\w<1$. Since $K_0$ and
$N_0$ are independent
of $\omega$ then $R_0$ (and hence $k_w$) can be the same for all
$\omega\in\Omega_\ep$, proving part of property $(U2)$.

For the  bounded distortion estimate
we need to show that there exists a constant $ K_\w> 0 $ such that
for any $ x, y $ belonging to an element $ U_\omega\in\mathcal
P_\omega $ with return time $ R_\omega $, we have
$$
\log\left|\frac{\det Df_\omega^{R_\omega}(x)}{\det
Df_\omega^{R_\omega}(y)}\right| \leq K_\w
\dist(f_\omega^{R_\omega}(x), f_\omega^{R_\omega}(y)).
$$
 By the chain rule
    \begin{align*}
\log\left|\frac{\det Df_\omega^{R_\omega}(x)}{\det
Df_\omega^{R_\omega}(y)}\right| & = \log\left|\frac{\det
Df_{\sigma^n(\omega)}^{R_\omega-n}(f_\omega^n(x))}{\det
Df_{\sigma^n(\omega)}^{R_\omega-n}(f_\omega^n(y))}\right| +
\log\left|\frac{\det Df_\omega^{n}(x)}{\det
Df_\omega^{n}(y)}\right|.
    \end{align*}

For the first term in this last sum we observe that by
Lemma~\ref{le:grow2} we have
$$
\log\left|\frac{\det
Df_{\sigma^n(\omega)}^{R_\omega-n}(f_\omega^n(x))}{\det
Df_{\sigma^n(\omega)}^{R_\omega-n}(f_\omega^n(y))}\right|
 \le
D_0 \dist(f_\omega^{R_\omega}(x),f_\omega^{R_\omega}(y)).
$$
For the second term in the sum above, we may apply
Corollary~\ref{co.distortion0} and obtain
 $$
\log\left|\frac{\det Df_\omega^{n}(x)}{\det
Df_\omega^{n}(y)}\right| \le
 C_0\dist(f_\omega^{n}(x),f_\omega^{n}(y))
 .$$
Also by Lemma~\ref{le:grow2} we may write
 $$\dist(f_\omega^{n}(x),f_\omega^{n}(y))\le K_0 \dist(f_\omega^{R_\omega}(x),f_\omega^{R_\omega}(y)).$$
We just have to take $K_\w=D_0+C_0K_0$ which, clearly, can be
uniformly chosen on $w$, completing property $(U2)$.

For the uniformity condition (U1), given $N>1$ and ${\varrho}>0$  we define, for every
$\w\in\Omega_\ep$, the sets
 $\{R_{\w}=j\}$, $A_j^\w$ and $B_j^\w$, with $j\leq
N$, as described in
 Section \ref{s.algoritmo}. The process that leads to the
 construction of these sets is based on the fact that small domains
 in $\Delta_0$ became large (in balls of radius $\delta_1$) by
 $f_{\w}^k$, for some $0\leq k\leq N_\alpha$, and then by
 $f^i_{\sigma^k(\w)}$, with $0\leq i\leq N_0$, they cover
 completely the ball $B(p,2\sqrt{\delta_0})\supset\Delta_0$.
 Hence, just by the continuity of $\Phi$, associated to the random perturbation
$\{\Phi,\{\theta_\ep\}_{\ep>0}\}$, we
 guarantee that, for any two realizations $\w$, ${\w'}$ in
 $\Omega_\ep$ the Lebesgue measure of the symmetric difference of respective
sets $\{R_{\w}=j\}$, $A_\w^j$ and $B_j$, for
 $j\leq N$, is smaller than ${\varrho}$, as long as we take $\ep$ sufficiently
 small. In particular, this holds for
 $\{R_{\sigma^{-j}(\w)}=j\}$ and $\{R_{\sigma^{-j}({\w'})}=j\}$, for
every
 $j=1,2,\ldots,N$.

\subsection{Metric estimates}
We compute now some estimates to show that the algorithm before
indeed
 produces a partition (Lebesgue mod 0) of $ \Delta_0 $.

\subsubsection{Estimates obtained from the construction}
In this first part we obtain some estimates relating the Lebesgue
measure of the sets $A_\w^n$, $B_\w^n$ and $\{R_\omega>n\}$ with the
help of specific information extracted from the inductive
construction we performed in Section~\ref{s.algoritmo}.

 \cle\label{l.flowb} There exists a constant $c_0>0$ (not depending on $\delta_0$)
  such that for every $\omega\in\Omega_\ep$ and
$n\ge1$  $$\leb(B_\w^{n-1}\cap A_\w^n)\ge
c_0\leb(B_\w^{n-1}).$$
 \fle

\dem
It is enough to see that this holds for each connected component
of $B_\w^{n-1}$ at a time. Let  $C$ be a component of $B_\w^{n-1}$ and
$Q$ be  its outer ring corresponding to $t_\w^{n-1}=1$. Observe that
by Remark~\ref{r.claim} we have $Q\subset C\cap A_\w^n$.
Moreover, there must be some $k<n$ and a component $U^0_k$ of
$\{R_\omega=k\}$ such that $f_\omega^k$ maps $C$
diffeomorphically onto $\cup_{i=k}^{+\infty} I_i$ and $Q$ onto
$I_k$, both with distortion bounded by $C_1$ and $e^{D_0D}$,
where $D$ is the diameter of $M$; cf. Corollary~\ref{co.distortion1} and
Lemma~\ref{le:grow2}. Thus, it is sufficient to compare the
Lebesgue measures of $\cup_{i=k}^{+\infty} I_i$ and $ I_k$. We have
 $$
 \frac{\leb( I_k)}{\leb(\cup_{i=k}^{+\infty}
I_i)}\thickapprox\frac{
[\delta_0(1+\lambda^{(k-1)/2})]^d-[\delta_0(1+\lambda^{k/2})]^d}
{[\delta_0(1+\lambda^{(k-1)/2})]^d-\delta_0^d}\thickapprox
1-\lambda^{1/2}.
 $$
 Clearly this proportion does not depend on  $\delta_0$ neither on $\omega$.
  \cqd

  \cle\label{l.flowa}   There exist $d_0,r_0>0$  with $d_0+r_0<1$
  such that for every $\omega\in\Omega_\ep$ and $n\ge1$
\begin{enumerate}
 \item $\leb(A_\w^{n-1}\cap B_\w^n)\le d_0\leb(A_\w^{n-1})$;
 \item $\leb(A_\w^{n-1}\cap \{R_\omega=n\})\le r_0\leb(A_\w^{n-1})$.
\end{enumerate}
Moreover $d_0\to 0$ and $r_0\to 0$  as $\delta_0\to 0$.
 \fle

  \dem It is enough to prove these estimates for each neighborhood
of a component $U^0_n$ of $\{R_\omega=n\}$. Observe that by
construction we have $U^3_n\subset A_\w^{n-1,\alpha}$, which means
that $U^2_n \subset A_\w^{n-1}$, because
$\alpha<\delta_0<\sqrt\delta_0$. Using the  distortion bounds of
$f_\omega^n$ on $U^3_n$ given by Corollary~\ref{co.distortion2} and
Lemma~\ref{le:grow2}  we obtain
 $$
 \frac{\leb(U^1_n\setminus U^0_n)}{\leb(U^2_n\setminus U^1_n)}
 \thickapprox
 \frac{\leb(\Delta^1_0\setminus \Delta_0)}{\leb(\Delta^2_0\setminus \Delta^1_0)}
 \thickapprox
 \frac{\delta_0^d}{\delta_0^{d/2}}\ll 1,
 $$
which gives the first estimate. Moreover,
$$
 \frac{\leb( U^0_n)}{\leb(U^2_n\setminus U^1_n)}
 \thickapprox
 \frac{\leb( \Delta_0)}{\leb(\Delta^2_0\setminus \Delta^1_0)}
 \thickapprox
 \frac{\delta_0^d}{\delta_0^{d/2}}\ll 1,
 $$
and this gives the second one.
 \cqd

We state a Lemma useful to prove the following proposition.

\cle\label{auxestimates} There exists $L$, depending only on
the manifold $M$, such that for every finite Borel measure
$\vartheta$ and every measurable subset $G \subset M$ with
compact closure there is a finite subset $I\subset G $ such
that the balls $B(z,\frac{\delta_1}4)$ around the points $z\in
I$are pairwise disjoint, and
$$\sum_{z\in
I}\vartheta\left(B\left(z,\frac{\delta_1}4\right)\cap G \right)\geq
L\vartheta(G )$$
\fle

\begin{proof} See \cite{A04}, Lemma 4.9.
\end{proof}

The next proposition asserts that a fixed proportion of
$A_\w^{n-1}\cap H^n_\omega$ gives rise to new elements of the
partition within a finite number of steps (not depending on $n$). We state first an auxiliary result.

\cpr\label{p.construction}
    There exist
$s_0>0$ and a positive integer $N=N(\alpha)$ such
    that for every $\omega\in\Omega_\ep$  and $ n\ge1$
    $$
     \leb\left(\bigcup_{i=0}^N\big\{R_\omega=n+i\big\}\right)\ge s_0
\leb(A_\w^{n-1}\cap H^{n}_\w).
    $$
    \fpr
\dem We use Lemma \ref{auxestimates} with
$G =f_\w^n(A_\w^{n-1}\cap H^n_\w)$ and $\vartheta=(f_\w^n)_*m$,
thus obtaining a finite subset $I$ of points $z\in
f_\w^n(A_\w^{n-1}\cap H^n_\w)$ for which the conclusion of the lemma
in particular implies
\begin{equation}\label{proportion}
\sum_{z\in I} m\left(\left((f_\w^n)^{-1}B\left(
z,\frac{\delta_1}{4}\right)\right)\cap A_\w^{n-1}\cap
H^n_\w\right)\geq L m(A_\w^{n-1}\cap H^n_\w).
\end{equation}

Fix now $z\in I$. Consider $\{C_j\}_j$ the set of connected
components of $(f_\w^n)^{-1}B(z,\delta_1/4)$, which intersect
$A_\w^{n-1}\cap H^n_\w$. Note that each $C_j$ is contained in a
hyperbolic preball $V^n_\w(x_j)$ associated to some point
$x_j\in((f_\w^n)^{-1}B(z,\delta_1/4))\cap A_\w^{n-1}\cap H^n_\w$ as
in Proposition~\ref{p.contr}. In what
follows, given $A\subset
B(z_j,\frac{\delta_1}4)$, we will simply denote
 $(f_\omega^{n}\vert_{V^n_\omega(x_j)})^{-1}(A)$ by
 $(f_\omega^{n})_j^{-1}(A)$. Note that the sets
 $\{(f_\omega^{n})_j^{-1}B(z,\frac{\delta_1}4)\}_j$ are pairwise
 disjoint as long as  that $\delta_1$ is sufficiently small (only depending on the
 manifold). In fact $f_\w^n$ sends each of them onto
 $B(z,\frac{\delta_1}{4})$ and $f_w^n$ is a diffeomorphism
 restricted to each one of them. In particular, their union does
 not contain points of $\mathcal{C}$.

 \medskip

\noindent {\sc Claim 1.} {\em There is $0\le k_j\le
N_\alpha+N_0$ such that $t_\w^{n+k_j}$ is not identically 0 in
$(f_\w^n)_j^{-1}B(z,\alpha)$}.

\medskip

Assume by contradiction that
  $
  t_\w^{n+k_j}\vert_{ (f_\omega^{n})_j^{-1}B(z,\alpha)}=0$ for all $0\le
k_j\le
  N_\alpha+N_0
  $.
 This implies that $(f_\omega^{n})_j^{-1}B(z,\alpha)\subset
A_\w^{n+k_j,\alpha}$
 for all $0\le k_j\le N_\alpha+N_0$.
 Using Lemma~\ref{le:grow1} we may find  a hyperbolic pre-ball
 $ V^{m}_{\sigma^n(\omega)}\subset B(z,\alpha) \subset A_\w^{n+k_j}$ with $
m\leq  N_{\alpha} $.
 Now, since $f_{\sigma^n(\omega)}^m(V^{m}_{\sigma^n(\omega)})$ is a ball of radius $\delta_1$ it follows
from Lemma~\ref{le:grow2} that there is some $V\subset
f_{\sigma^n(\omega)}^m(V^{m}_{\sigma^n(\omega)})$
and $m'\le N_0$ with $f_{\sigma^{n+m}(\omega)}^{m'}(V)=\Delta_0$.
Thus, taking $k_j=m+m'$ we have that $0\le k_j\le
N_\alpha+N_0$ and $(f_{\omega}^{n})_j^{-1}(V^m_{\sigma^n(\omega)})$
contains an element of $\{R_\omega=n+k_j\}$ inside
$(f_\omega^{n})_j^{-1}B(z,\alpha)$. This contradicts the fact
that $t_\w^{n+k_j}\vert_{ (f_\omega^{n})_j^{-1}B(z_j,\alpha)}=0$
 for all $0\le k_j\le N_\alpha+N_0$.

\medskip

\noindent {\sc Claim 2.} {\em $(f_\w^n)_j^{-1}B(z,\delta_1/4)$
contains a component $\{R_w=n+k_j\}$, with $0\le k_j\le
N_\alpha+N_0$.}

\medskip

Let $k_j$ be the
 smallest integer $0\le k_j\le N_\alpha+N_0$ for which
$t_\w^{n+k_j}$ is not identically zero in
$(f_\omega^{n})_j^{-1}B(z,\alpha)$.
Since
  $(f_\omega^{n})_j^{-1}B(z,\alpha)\subset A_\w^{n-1,\alpha}\subset
\{t_\w^{n-1}\le1 \},
  $
 there must be some element
$U^{0}_{n+k_j}$ of $\{R_\omega=n+k_j\}$ for which
 $
((f_\omega^{n})_j^{-1}B(z,\alpha))\cap
U_{n+k_j}^1\neq\emptyset.
 $
Recall that by definition $f_\omega^{n+k_j}$ sends $U_{n+k_j}^1$
diffeomorphically onto $\Delta_0^1$, the ball of radius
$2\delta_0$ around $p$. From time $n$ to $n+k_j$ we may have some
final ``bad" period of length at most $N_0$ {where the derivative
of $f$ may contract}, however being bounded from below by $1/K_0$
in each step. Thus, the diameter of $f_\omega^n(U_{n+k_j}^1)$ is
at most $4\delta_0K_0^{N_0}$. Since $B(z,\alpha)$ intersects
$f_\omega^n(U_{n+k_j}^1)$ and
$\alpha<\delta_0<\delta_0K_0^{N_0}$, we have
 $
( f_\omega^{n})_j^{-1}B(z,\delta_1/4)\supset U_{n+k_j}^0,
 $
as long as we take $\delta_0>0$ small enough so that
\begin{equation*}
5\delta_0K_0^{N_0}<\frac{\delta_1}{4}.\end{equation*}
Thus, we have shown  that $(f_\omega^{n})_j^{-1}B(z,\delta_1/4)$
contains some component of $\{R_\omega=n+k_j\}$ with  $0\le k_j\le
N_\alpha+N_0$, and proved the claim.

\medskip

Since $n$ is a hyperbolic time for each $x_j$, we have by the
distortion control given by Corollary~\ref{co.distortion2}
  \begin{equation}
\label{eq.quo2}
\frac{\leb((f_\omega^{n})_j^{-1}B(z,\delta_1/4))}{\leb(
U_{n+k_j}^0)}
 \le
 {C_2}\frac{\leb(B(z,\delta_1/4))}{\leb(f_\omega^n(U_{n+k_j}^0))}.
\end{equation}

From time $n$ to time $n+k_j$ we have at most $k_j=m_1+m_2$
iterates with $m_1\le N_\alpha$, $m_2\le N_0$ and
$f_\omega^n(U_{n+k_j}^0)$ containing some point $y_j\in
H^{m_1}_{\sigma^n(\omega)}$. By the definition of
$(\lambda,\delta)$-hyperbolic time we have $\dist_\delta
(f_{\sigma^{n}(\omega)}^i(x),\cc)\ge \lambda^{bN_\alpha}$ for
every $0\le
i\le m_1$, which 
implies that there is some constant $D_1=D_1(\alpha)>0$ such that
$|\det (Df_{\sigma^n(\omega)}^i(x))|\le D_1$ for $0\le i\le m_1$ and
$x\in f_\omega^n(U_{n+k_j}^0)$. On the other hand, since the first
$N_0$ preimages of $\Delta_0$ are uniformly bounded away from
$\cc$ we also have some $D_2>0$ such that $|\det
(Df_{\sigma^{n+m_1}(\omega)}^i(x))|\le D_2$ for every $0\le i\le
m_2$ and $x$ belonging to an $i$-preimage
$(f_{\sigma^{n+k_j-i}(\omega)}^i)^{-1}\Delta_0$ of $\Delta_0$. Hence,
$$\leb(f_\omega^n(U_{n+k_j}^0))\ge \frac{1}{D_1D_2}\leb(\Delta_0),$$
which combined with (\ref{eq.quo2}) gives
 $$
 \leb\left((f_\omega^{n})_j^{-1}B\left(z_j,{\delta_1}/{4}
\right)\right)\le D\leb(
U_{n+k_j}^0),
 $$
with $D$ only depending on $C_2$, $D_1$, $D_2$, $\delta_0$ and
$\delta_1$. Moreover, if $\ep$ is small enough, $D_1$ and $D_2$ can be taken uniform over $\w$.

We are now able to compare the Lebesgue measures of
$\cup_{i=0}^N\big\{R_\omega=n+i\big\}$ and $A_\w^{n-1}\cap H^{n}_\w$.
Using \eqref{proportion} we get
$$
 \begin{array}{ll}
 \leb(A_\w^{n-1}\cap H^n_\w)&\leq L^{-1}\displaystyle\sum_{z\in I}
\leb\left((f_\omega^{n})_j^{-1}B\left(z_j,{\delta_1}/{4}
)\right)\cap A_\w^{n-1}\cap
 H^n_\w\right)\\
 &\leq L^{-1}\displaystyle\sum_{z\in I}\sum_j
\leb\left((f_\omega^{n})_j^{-1}B\left(z_j,{\delta_1}/{4}
\right)\right)\\
&\leq DL^{-1}\displaystyle\sum_{z\in I}\sum_j \leb(U_{n+k_j}^0)
 \end{array}
$$
One should mention that the sets $U_{n+k_j}^0$ also depend on $z\in
I$. They are disjoint for different values of $z\in I$. Hence, putting
$N=N_0+N_\alpha$, we have
$$\leb(A_\w^{n-1}\cap H^{n}_\w)\leq
DL^{-1}\leb\left(\cup_{i=0}^N\big\{R_\omega=n+i\big\}\right).$$

To finish the proof we only have to take $s_0=DL^{-1}$.\cqd

\cre \label{re.c1}It follows from the choice of the constants $D_1$
and $D_2$ and  $D$  that the constant $s_0$ only depends on the
constants $\lambda$, $b$, $N_\alpha$, $N_0$, $C_2$,
$\delta_0$ and $\delta_1$. Recall that $L$ is an absolute
constant only depending on $M$.\fre

\subsubsection{Independent metric estimates}\label{s.check}

 We have taken a disk $\Delta_0$ of radius
   $\delta_0>0$ around a point $p\in M$ with certain properties
   and, for every $\omega\in\Omega_\ep$, we defined inductively the subsets
  $A_\w^n$, $B_\w^n$, $\{R_\omega=n\}$ and $\Lambda_\w^n$ which are related in
   the following way:
    $$\Lambda_\w^n=\Delta_0\setminus\{R_\w\le n\}=A_\w^n\dot\cup B_\w^n.$$
Since we are dealing with a non-uniformly expanding on random
orbits system,  for each $\omega\in\Omega_\ep$ and each $n\in \NN$
we also have defined the set  $H_\w^n\subset M$ of points that
have $n$ as a $(\lambda,\delta)$-hyperbolic time, and the tail set
$\Gamma_\omega^n$  as in \eqref{tailset}. From the definition of
$\Gamma_\omega^n$, Remark~\ref{r.start} and
Lemma~\ref{c:hyperbolic3} we deduce that for every $\w\in \Omega_\ep$:

\begin{enumerate}
\item[(m$_1$) ] there is $\zeta>0$ such that for every $\w\in\Omega_\ep$, $n\ge1$
and every $A\subset M\setminus\Gamma_\omega^n$ with $m(A)>0$
    \[
   \frac{1}{n}\sum_{j=1}^{n}\frac{\leb(A\cap H_\w^j)}{\leb(A)} \geq \zeta.
   \]
\end{enumerate}
Moreover, Lemmas~\ref{l.flowb},
Lemma~\ref{l.flowa} and Proposition~\ref{p.construction} give us a {random} version of metric relations (m$_2$)-(m$_4$) for Section 4.5.2 in \cite{A04}.
In the inductive process of construction of the sets $A_\w^n$, $B_\w^n$,
$\{R_\omega=n\}$ and $\Lambda_\w^n$ we have fixed some large integer
$R_0$, this being the first step at which the construction began.
Recall that $A_\w^n=\Lambda_\w^n=\Delta_0$ and
$B_\w^n=\{R_\omega=n\}=\emptyset$ for $n\le R_0$. We will assume that
 \begin{equation*}
 R_0>\max\left\{2(N+1),{12}/\zeta\right\}.
 \end{equation*}
Note that since $N$ and $\zeta$ do not depend on $R_0$ this is
always possible, so we can follow Section 4.5.2 at \cite{A04} to conclude
that for
every $\w\in\Omega_\ep$ this process indeed produces a partition
$\mathcal{P}_\w=\{R_\w=n\}_n$ of $\Delta_0$. Moreover, it also follows from there that, if there
exist $C,
p>0$ such that for every $\w\in\Omega_\ep$ we have
$\leb(\Gamma_\w^n)\leq Cn^{-p}$ then there exists $C'$ such that for every
$\w\in\Omega_\ep$
 the return time function
satisfies
\begin{equation}
\leb(R_\w>n)\leq C'n^{-p}.
\end{equation}
It is possible to check that constant $C'$ depend ultimately on
the constants $B$, $\beta$ and ${b_0}$ associated to the
non-uniform expanding condition in Definition \ref{defNUEOA}.
This implies that $C'$ can be considered the same for every
$\w\in\Omega_\ep$.

\section{Applications}\label{se.examples}

\subsection{Local diffeomorphisms}\label{local difeos}

One example of transformations that fits our hypothesis was introduced in \cite{ABV00} and consists on
robust ($C^1$ open) classes of local diffeomorphisms (with no
critical sets) that
 are non-uniformly expanding. The existence and unicity of SRB pro\-ba\-bi\-li\-ty measures for this maps was proved in \cite{ABV00} and \cite{A03}. Random perturbations for this maps were considered
in \cite{AAr03}, where it was proved a weak form of stochastic stability - the
convergence in the weak$^*$ topology of the density of the unique stationary
probability measure to the density of the unique SRB probability measure. Here we improve it to the strong
version
of stochastic stability. As corollary we also obtain the strong statistical stability, proved in \cite{A04}. We follow closely the constructions and results in \cite{ABV00}
and
\cite{AAr03} and introduce some extras to have the required transitivity.

 This classes of maps  and can be obtained, e.g.
through
deformation of a uniformly expanding map by isotopy inside some
small region. In general, these maps are not uniformly expanding:
deformation can be made in such way that the new map has periodic
saddles.

 Let $M$ be the  $d$-dimensional torus $\mathbb{T}^d$, for some $d\ge 2$, and $m$
  the normalized Riemannian volume form. Let  $f_0\colon M\to M$ be a uniformly expanding map and $V\subset M$ be a small
neighborhood of a fixed point $p$ of $f_0$ so that the restriction of $f_0$ to $V$ is
injective. Consider a $C^1$-neighborhood $\cu$ of $f_0$ sufficiently small so that
any map $f\in\cu$ satisfies:
\begin{enumerate}
\item[i)] $f$ is {\em expanding outside $V$}: there
exists $\lambda_0<1$ such that
 $$\|Df(x)^{-1}\|< \lambda_0\quad\text{for
every $x\in M\setminus V$};$$
 \item[ii)] $f$ is {\em volume expanding
everywhere}\index{expanding!volume}: there exists $\lambda_1>1$
such that
 $$|\det Df(x)| > \lambda_1\quad\text{ for every $x\in M$;}$$

 \item[iii)]
 $f$ is {\em not too contracting on~$V$}: there is some
 small $\gamma>0$ such that
$$\|Df(x)^{-1}\|< 1+\gamma\quad\text{ for every $x\in
V$},$$
 \end{enumerate}
 and constants $\lambda_0, \lambda_1$ and $\gamma$ are the same for all
$f\in\cu$.
 Moreover, for $f\in\cu$ we introduce random perturbations
$\{\Phi,(\theta_{\ep})_{\ep>0}\}$. In particular,
we consider a
continuous map
  \[
  \begin{array}{rccl}
  \Phi:& T &\longrightarrow&  \cu\\
  & t &\longmapsto & f_t
  \end{array}
  \]
 where $T$
is a metric space and $f\equiv f_{t^*}$ for some
$t^*\in T$. Consider
a family
$(\theta_\ep)_{\ep>0}$ of
probability measures on $T$ such that their supports are non-empty and satisfies $
 \supp(\theta_\ep)\rightarrow \{t^*\}$, {when} $\ep\to 0$.
 We can choose appropriately the constants
$\lambda_0$, $\lambda_1$ and $\gamma$ so
that every map $f\in\cu$ is non-uniformly
expanding on {\em all} random orbits with uniform exponential decay of the Lebesgue measure of the tail sets $\Gamma_\w^n$ given by \eqref{tailset}, ignoring naturally the recurrence time function.
 \begin{Proposition}\label{nuedifeos}
 Consider $f_0$, $\mathcal U$, $f\in\cu$ and $\{\Phi,(\theta_{\ep})_{\ep>0}\}$
as before. There exists
$a_0>0$ such that for every $\w\in \supp(\theta_\ep^\NN)$ and Lebesgue almost every $x\in M$
\begin{equation}\label{nueoadifeos}\limsup_{n\to+\infty}\frac1n\sum_{j=0}^{n-1}
\log\| Df_{\sigma^j(\w)}(f^j_{\w}(x))^{-1}\| \le
-a_0.\end{equation}
Moreover, there is $0<\tau<1$ such that
 $m(\Gamma_\w^n)\leq \tau^n$,
for $n\ge 1$ and $\w\in \supp(\theta_\ep^\NN)$.
 \end{Proposition}
 \begin{proof} See \cite{AAr03}.

\cqd

We now show that performing the construction a bit more carefully we have the maps also transitive, which in particular implies that each map has a unique SRB probability measure. We shall actually prove that those maps are topologically mixing.
We start by considering a map $\bar f\colon M\to M$ (in the boundary of the set of uniformly expanding maps) which satisfies \emph{(i), (ii)} and \emph{(iii)} as the cartesian product of  one-dimensional maps $\varphi_1\times\cdots\times\varphi_d$, with $\varphi_1,\dots,\varphi_{d-1}$ uniformly expanding in $S^1$, and $\varphi_d$ the \emph{intermittent} map in $S^1$: it can be written as
 $$\varphi_d(x)=x+x^{1+\alpha}, \quad \text{for some $0<\alpha<1$,}$$
in a neighborhood of 0  and $\varphi_d'(x)>1$ for every $x\in S^1\setminus\{0\}.$
 One already has that  any $f$ in a sufficiently small $C^1$-neighborhood
 $\bar\cu$ of $\bar f$ satisfies \emph{(i), (ii)} and \emph{(iii)} for convenient choice of constants
 $\lambda_0$, $\lambda_1$, $\gamma$, and neighborhood $V$ of the fixed point $p=0\in \TT^d$.
 Next lemma ensures that $\bar f$ is topologically mixing, and thus topologically transitive. 
We show moreover that if $\bar\cu$ is sufficiently small, then all the maps in $\bar\cu$ are topologically mixing.

\cle\label{le.LEO}
Given  $\alpha>0$ there is $N_\alpha\ge 1$  such
that $\bar f^{N_\alpha}(B_\alpha(x))=\TT^d$ for any $x\in M$. \fle

\begin{proof} This is an immediate consequence of the fact that a  similar conclusion holds for the maps $\varphi_1,\dots,\varphi_d$ in $S^1$. This is standard for the uniformly expanding maps $\varphi_1,\dots\varphi_{d-1}$, and also for the intermittent map $\varphi_d$ as it is topologically conjugate to a uniformly expanding map of the same degree.
\end{proof}

Let us now obtain a similar conclusion for any map $f$ in $\bar\cu$. This cannot be done by a simple continuity argument, since for smaller radii $\alpha>0$ in principle we need to diminish the size of the $C^1$-neighborhood. However, a continuity argument works if one just needs to consider balls of some fixed radius. By Proposition \ref{nuedifeos} any map $f\in\bar\cu$  is non-uniformly expanding and, if we consider a random perturbation of $f$ as before, then $f$ is also non-uniform expanding on (all) random orbits (naturally, for sufficiently small noise level), with uniform exponential decay of the Lebesgue measure of the tail set. By Propositions \ref{nuedifeos} and \ref{pr.hyperbolic1}, Lebesgue almost every point $x\in M$ has infinitely many $\lambda$-hyperbolic times and,  moreover, we may take $\lambda=e^{-a_0/2}$. Lemmas \ref{l.contr} and \ref{p.contr} imply that there exists $\delta_1>0$ (uniform for the maps in $\bar\cu$) such that Lebesgue almost every point in $\TT^d$ has arbitrarily small neighborhoods which are sent onto balls of radius $\delta_1>0$. Taking $\alpha=\delta_1/2$ in Lemma~\ref{le.LEO}, there is some  positive integer $N$ for which every ball of radius $\delta_1/2$ is sent onto $M$ by $\bar f^N$. Then, just by continuity, one has that any ball of radius $\delta_1$ is sent onto $M$ by $f^N$ for any $f\in\bar\cu$, provided this  $C^1$-neighborhood is sufficiently small. Then, in particular, each $f\in\bar\cu$ is topologically transitive. The next theorem is now a direct application of Theorem \ref{thA}.

\cte  Let $f\in\bar\cu$. Then
\begin{enumerate}
    \item if $\ep$ is small enough then $
f$
admits a unique absolutely continuous
ergodic stationary
probability measure;
    \item $f$ if
strongly stochastically stable.
  \end{enumerate}
\fte

%
%

\subsection{Viana maps}

We consider now an important open class of non-uniformly expanding maps with
critical sets in higher
dimensions introduced in \cite{V97}. This example features the hypothesis of Theorem \ref{thA}
resulting on the proof of their strong stochastic
stability.
The existence of a unique  absolutely continuous ergodic invariant probability
measure and the strong
statistical stability were proved in \cite{AV02}. A weaker form of stochastic
stability (weak$^*$ convergence of the stationary measure to
$\mu_f$) was established in \cite{AAr03}.
In order to check the hypothesis of Theorem \ref{thA} we use essentially the results in \cite{AAr03}
 about the non-uniform
expansion, slow recurrence to the critical
set and uniform decay of the Lebesgue measure of the tail set, both for deterministic and random cases.
Without loss of generality we discuss the
two-dimensional
case and we refer \cite{V97},  \cite{AV02} and \cite{AAr03} for details.\\

Let $p_0\in(1,2)$ be such that the
critical point
$x=0$ is pre-periodic for the quadratic map $Q(x)=p_0-x^2$. Let
$S^1=\RR/\ZZ$ and $b:S^1\rightarrow \RR$ be a Morse function, for
instance, $b(s)=\sin(2\pi s)$. For fixed small $\alpha>0$,
consider the map
 $$ \begin{array}{rccc} \hat f: & S^1\times\RR
&\longrightarrow & S^1\times \RR\\
 & (s, x) &\longmapsto & \big(\hat g(s),\hat q(s,x)\big)
\end{array}
 $$
 where $\hat g$ is the uniformly expanding
map of the circle defined by $\hat{g}(s)=ds$ (mod $\ZZ$) for some
$d\ge16$, and $\hat{q}(s,x)=a(s)-x^2$ with $a(s)=p_0+\alpha b(s)$. As it is shown in \cite{AAr03}, it
is no restriction to assume that $\cc=\{(s,x)\in S^1\times I\colon
x=0\}$ is the critical set of $\hat f$  and we do so.
If
$\alpha>0$ is small enough there is an
interval $I\subset (-2,2)$ for which $\hat f(S^1\times I)$ is
contained in the interior of $S^1\times I$. Any map $f$
sufficiently close to $\hat f$ in the $C^3$ topology has
$S^1\times I$ as a forward invariant region (in fact, here it suffices to be
$C^1$
close). We consider a small $C^3$ neighborhood $\mathcal V$ of $\hat f$ as
before and will refer to maps in $\mathcal V$ as {\em Viana maps}. Thus, any
{\em Viana map} $f\in\mathcal V$ has $S^1\x I$ as a forward invariant region,
and so an {\em attractor} inside it, which is precisely $$\Lambda=\bigcap_{n\geq
0}f^n(S^1\x I).$$

 We introduce the random perturbations $\{\Phi,(\theta_\ep)_\ep\}$ for this maps. We set
   $T\subset\mathcal V$
to be a $C^3$
neighborhood of $\hat{f}$ consisting in maps $f$ restricted to the
forward invariant region
$S^1\times I$ for which
$Df(x)=D\hat f(x)$ if $x\notin\mathcal C$, the map $\Phi$ to be the identity map at $T$ and
$(\theta_\ep)_{\ep}$ a family of Borel measures on $T$ such that their supports are non-empty and satisfy $
 \supp(\theta_\ep)\rightarrow \{f\}$, {when} $\ep\to 0$, for  $f\in T$. Let $h_{f}$ to be the density
of the
unique absolutely continuous invariant probability measure $\mu_{f}$ for $f$.
We will show that such Viana maps $f\in\mathcal V$ satisfies the hypothesis of Theorem
\ref{thA} so that we may conclude

 \cte\label{SSSViana}
  Let $f\in\mathcal V$ be a Viana map. Then
  \begin{enumerate}
    \item if $\ep$ is small enough then $
f$
admits a unique absolutely continuous
ergodic stationary
probability measure,
    \item $f$ is
strongly stochastically stable.
  \end{enumerate}
\fte

%
%

\subsubsection{Deterministic estimates}

 The results in
\cite{V97} show that if $\mathcal V$ is sufficiently small (in the $C^3$
topology)
then $f\in\mathcal V$ has two positive Lyapunov
exponents almost everywhere:
 there is a constant $\eta>0$ for which
$$\liminf_{n\rightarrow+\infty}\frac{1}{n}\log
\|Df^n(s,x)v\|\geq \eta$$ for Lebesgue almost every $(s,x)\in
S^1\times I$ and every non-zero  $v\in T_{(s,x)}(S^1\times I)$.
This does not necessarily imply that $f$ is non-uniformly
expanding. However,  as it was shown in \cite{AAr03}, a slightly deeper use of Viana's arguments
enables us to prove the non-uniform expansion and the slow recurrence to the critical set of any $C^2$ map $f$ such that
$$\|f-\hat f\|_{C^2}<\alpha.$$
In particular they proved that there exist $C, \zeta>0$ such that for  $f$ as before and $n\geq1$ there is a set $\Gamma_n\subset S^1\times I$ for which
 \begin{equation*}
 m(\Gamma_n)\le C e^{-\zeta\sqrt n},
 \end{equation*}
and such that
  for each $(s,x)\notin \Gamma_n$ we have\begin{enumerate}
\item  there is $a_0>0$ such that \begin{equation}
 \nonumber
 \frac1k\sum_{j=0}^{k-1}\log\|Df(f^j(s,x))^{-1}\|\le
 - a_0
 \quad\mbox{for all } k\ge n,\text{ and}
 \end{equation} \item for given small $b_0>0$ there is $\delta>0$ such that \begin{equation}
 \nonumber
 \frac1k\sum_{j=0}^{k-1}-\log\dist_\delta(f^j(s,x),\cc)\leq b_0
 \quad\mbox{for all } k\ge n.
 \end{equation}
 \end{enumerate}
Moreover, the constants $\zeta$, $a_0$ and $\delta$ only depend on the quadratic map $Q$ and $\alpha>0$.

In \cite{AV02} it was also proved
a topological mixing property.
\cte
For every $f\in\mathcal V$ and every open set $A\subset
S^1\x I$ there is some $n_A\in\NN$ for which $f^{n_A}(A)=\Lambda$.
\fte
\begin{proof}
See \cite{AV02}, Theorem C.
\end{proof}

\subsubsection{Estimates for random perturbations} Let $f$ be close to $\hat f$ in the $C^3$ topology and fix a random perturbation $\{\Phi,(\theta_\ep)_{\ep>0}\}$ as before. We want to show that if $\ep>0$
is small enough then
  $f$ is non-uniformly expanding on random orbits and $\Gamma_\omega^n$ decays sufficiently fast and uniformly on $\omega$. The estimates in \cite{AAr03} for
$\log\dist_\delta(f^j(s,x),\cc)$ and $\log\|Df(f^j(s,x))^{-1}\|$ over
the orbit of a given point $(s,x)\in S^1\times I$ can easily be done replacing
the iterates $f^j(s,x)$  by random iterates $f_\w^j(s,x)$. Briefly, those estimates rely on a delicate
decomposition of the orbit of the point $(s,x)$ from time 0 until
time $n$ into finite pieces according to its returns to the
neighborhood $S^1\times (-\sqrt\alpha,\sqrt\alpha)$ of the critical set.
The main tools for this estimates were \cite[Lemma~2.4]{V97} and \cite[Lemma~2.5]{V97}
whose proofs may easily be mimicked for random orbits. Indeed, the
important fact in the proof of the referred lemmas is that orbits
of points in the central direction stay close to  orbits of the
quadratic map $Q$ for long periods, as long as $\alpha>0$ is taken
sufficiently small. Hence, such results can easily be obtained for
random orbits as long as we take $\ep>0$ satisfying $\ep\ll\alpha$. It was also proved in \cite{AAr03} that exists $C>0,\zeta>0$ such that
$\leb(\Gamma_\w^n)<Ce^{-\zeta\sqrt n}$, for almost every $\w\in\supp(\theta_\ep^\NN)$, which clearly is enough for our purposes. Moreover, the constants for the estimates on the the tail set, non-uniform expansion and slow recurrence remains depending only on the quadratic map $Q$ and $\alpha$. In particular, they are uniform over $\w$.

\bibliographystyle{novostyle}


\end{document}